%% file: motors_arxiv.tex
\documentclass{article}

\usepackage{fullpage}
\usepackage[latin1]{inputenc}
\usepackage[T1]{fontenc}
\newcommand{\citep}{\cite}
\newcommand{\citet}{\cite}
\usepackage{wrapfig}
\usepackage{amsmath,amsthm}
\usepackage{amssymb}
\usepackage{latexsym}
\usepackage{graphicx,epstopdf}
\usepackage{enumerate}
\graphicspath{{./images/}}
\usepackage[rflt]{floatflt}

\def\rbb{\mathbb{R}}

\def\ep{\epsilon}

\def\Jmat{\mathsf{\mathit{J}}}
\def\Fstall{F_*}

\newtheorem{theorem}{Theorem}[section]
\newtheorem{prop}[theorem]{Proposition}

\newtheorem{lemma}[theorem]{Lemma}

\newcounter{assume}
\newtheorem{assumption}[assume]{Assumption}

\newcommand{\E}[1]{\mathbb{E}\left[#1\right]}

\def\tt{\tilde t}
\def\Phip{\Phi^{\prime}}

\input{preamble}

\begin{document}

\title{Asymptotic Analysis of Microtubule-Based Transport by Multiple Identical Molecular Motors}

\author{Scott A.~McKinley\thanks{Mathematics Department, University of Florida, Gainesville, FL, 32611}, Avanti Athreya\thanks{Department of Applied Mathematics and Statistics, Johns Hopkins University, Baltimore, MD 21218.}, 
John Fricks \thanks{Department of Statistics, Penn State University, University Park, PA 16802.} and Peter R.~Kramer\thanks{Department of Mathematical Sciences, Renssalaer Polytechnic Institute, Troy, NY 16802.}}

\maketitle









\begin{abstract}
	We describe a system of stochastic differential equations (SDEs) which model the interaction between processive molecular motors, such as kinesin and dynein, and the biomolecular cargo they tow as part of microtubule-based intracellular transport.  We show that the classical experimental environment fits within a parameter regime which is qualitatively distinct from conditions one expects to find in living cells.  Through an asymptotic analysis of our system of SDEs, we develop a means for applying \emph{in vitro} observations of the nonlinear response by motors to forces induced on the attached cargo to make analytical predictions for two parameter regimes that have thus far eluded direct experimental observation: 1) highly viscous \emph{in vivo} transport and 2) dynamics when multiple identical motors are attached to the cargo and microtubule.
\end{abstract}



\section{Introduction}

Critical to the proper functioning of a biological cell is the efficient internal transport of organelles and other intracellular cargo that collectively form the basis of the cellular infrastructure \cite{Howard:2001}. In eukaryotic cells, which are characteristic of all complex organisms, these biological materials are generally assembled near the nucleus and are then packaged into membrane-bound vesicles to be distributed through the cytoplasm to appropriate locations throughout the cell.  For larger vesicles and organelles (on the order of 1 $\mu m$), diffusion is too slow for efficient transport, and inadequate for proper spatial distribution of important organelles and molecules throughout the cell.  Cells have therefore evolved an intricate transport apparatus consisting of a cytoskeletal network of thin filaments called \emph{microtubules} and proteins called \emph{processive molecular motors} that move in a directed fashion along the microtubules while generating sufficient force to tow vesicles, which we will generally call \emph{cargo}, at a considerable rate (as high as 1 \emph{$\mu$m}/\emph{s} \cite{Hancock:2010,Visscher:1999,hirokawa1998kinesin}). 

The class of motor proteins we consider in this work are \emph{kinesins}, which are typically responsible for transport toward the cell periphery.
Two decades of biochemical and biophysical investigations have lead to a working model for how motor proteins move along microtubules \cite{Muthukrishnan:2009p353}: Kinesins are dimeric molecules consisting of two heads that attach to a microtubule, a coiled-coil tether that joins the two chains, and a cargo binding tail (Figure \ref{fig:transport}) \cite{hancock2003kinesin}.  Each head contains both an ATP and a microtubule binding site, and these motors ``walk'' along the microtubule track by chemical coordination of ATP hydrolysis cycles such that the cycles remain out of phase and at least one head remains bound to the microtubule.
We say the motors are ``processive'' because they take many steps during each encounter with a microtubule. 
All told, we have substantial information about the operation and properties of individual molecular motors bound to cargo through work done by biologists~\citep{Banting:2000,Kolomeisky:2007,Howard:2001,howard:1989}, physicists~\citep{Astumian:2002,Reimann:2002,Hanggi:2009,Maes:2003}, and mathematicians~\citep{Mogilner:2002,Wang:2008,Peskin:1995,DeVille:2008,Qian:2000,Kinderlehrer:2002}.  
\begin{wrapfigure}{r}{0.35\textwidth}
\includegraphics[scale=.6]{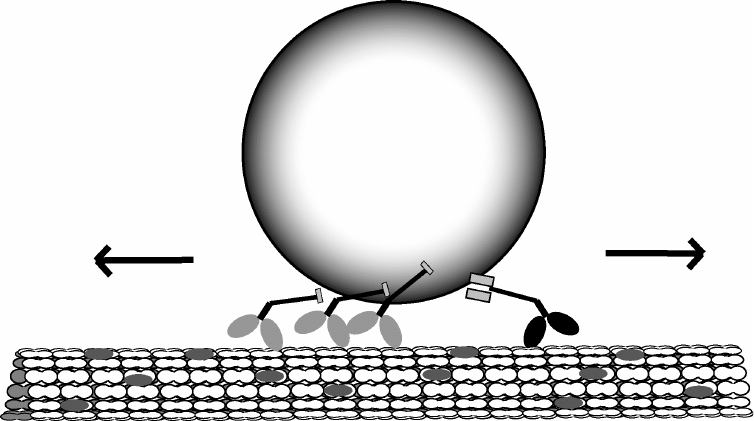} \centering
\caption{A cartoon model of intracellular transport. Intracellular cargo may be attached to microtubules by one or more molecular motors which individually tend to walk toward either the plus-end or minus-end of a microtubule.}
\label{fig:transport}
\end{wrapfigure}

However, translating these experimental observations and the associated theory for single motor and cargo systems into a predictive model for \emph{in vivo} intracellular transport has been constrained by several practical limitations.  First, the two primary experimental techniques, described in Section \ref{sec:discussion}, are conducted in fluid environments that have substantially lower viscosity than what is expected in the cytoplasm \cite{Mitchell:2009}. Second, it is expected that the \emph{in vivo} transport of a given cargo will involve multiple motors, possibly of different types, that attach and detach from microtubules dynamically.  Indeed, in their pioneering Markov chain model ~\cite{Muller:2008pnas,Muller:2008jstatphys,Muller:2010}, M\"uller, Klumpp and Lipowksy indicate the remarkably complex behavior that can result from a cargo responding to molecular motors arranged in a ``tug-of-war" configuration.  However, with one notable exception where motors were artificially bound to the cargo in pairs \cite{jamison:2010}, it is experimentally difficult to infer precisely how many motors of each type are bound to both the cargo and a microtubule at the same time.  
In recent years, therefore, theoretical work and numerical simulations of multiple-motor-cargo systems have both played a strong role in generating predictions that have been later verified, or modified as appropriate, by experimental findings~\citep{Klumpp:2005,Muller:2008pnas,Muller:2008jstatphys,Guerin:2010,Wang:2009,Campas:2006,Posta:2009,Korn:2009,Muller:2010,Zhang:2010,jamison:2010,kunwar2008stepping,Hendricks:2010,Rogers:2009,Larson:2009,Gross:2007,Vershinin:2007}. 

\subsection{Outline of the Paper}
Our goal is to contribute to these efforts by constructing a stochastic dynamical system that is both analytically tractable and rich enough to model the diverse conditions relevant for microtubule-based intracellular transport.  Our system of stochastic differential equations (SDEs) very naturally incorporates the following qualitative features of motor dynamics: a linear or nonlinear spring model for the motor-cargo linkage, thermal fluctuations that influence the cargo, a nonlinear response by the motor to cargo forces, nuanced changes in behavior due to motor configuration, and fluctuations in the motor position arising from the stochastic timing of the various mechanical and chemical events that occur during motor stepping. These features were absent in the original MKL models~\cite{Muller:2008pnas,Muller:2008jstatphys}, but are now common, not always in combination, in recent simulation-based studies \cite{Bouzat:2010,Wang:2009,Korn:2009,Kunwar:2010}. We find that they substantially affect the predictions of the model.  

The framework we develop is quite general, but we restrict the model considerably in order to perform specific rigorous analysis. In particular, we study the transport of thermally fluctuating cargo by identical cooperative motors in the absence of motor binding and unbinding.  This leads to certain results that are in contrast to prior studies, but when this happens, our analysis suggests that such properties may be a consequence of attachment/detachment dynamics which warrant separate study (see for example, \cite{Muller:2008pnas}). 
Our purpose in neglecting binding and unbinding dynamics in the present study is not to dispute their relevance, but to provide a point of reference for understanding how the various properties of the motors and cargo affect their collective behavior.

The core of our approach lies in Section \ref{sec:nondim}, where we perform a dimensional analysis of the system of SDEs with physically relevant parameters.  
We find that there is a crossover in qualitative behavior between the diffusion-dominated, low viscosity \emph{in vitro} environment in which many experiments are conducted, and the high viscosity environment \emph{in vivo}.  This result follows from the identification of critical nondimensional parameter groups summarized in Table \ref{table:nondim}.  In the low viscosity regime, the rapidly fluctuating cargo achieves a quasi-stationary distribution with respect to the slower moving motors; in Section \ref{sec:nondim} we introduce the proper stochastic averaging techniques to calculate the influence of the cargo on the motors, see \eqref{eq:motavg}.  In Section \ref{sec:low-viscosity} we demonstrate that this separation of time scales results in a motor-cargo system velocity that is insensitive to order-of-magnitude changes in the viscosity of the fluid environment. Indeed, this has been observed \emph{in vitro} by Gross et al \cite{Gross:2007} and \emph{in silico} by Kunwar and Mogilner \cite{Kunwar:2010}.  This framework also allows us to easily predict the result of perturbing various other parameters.  For example, with the development of synthetically modified motors \cite{Rogers:2009,yildiz2008intramolecular,Shastry:2010} it may soon be possible to tune the spring constant of the tail of the motor which binds to the cargo.  From equation \eqref{eq:velocity-one-motor} we observe that as the spring constant increases, the motor-cargo system velocity will decrease.

Recently there has been some discussion about how to translate inferences made from low viscosity \emph{in vitro} experiments into predictions more like the conditions of a biological cell \cite{Gross:2007,Shubeita:2008,Mitchell:2009}. The cytoplasm is highly viscous, and likely to have some viscoelastic properties. It is beyond the scope of this paper to address diffusion in a viscoelastic environment, so we adopt the same approximation for cytoplasmic viscosity as presented by Mitchell \& Lee \cite{Mitchell:2009}. In Section \ref{sec:onefluc} we attempt to establish an equivalence between the drag force felt by the motor via the cargo at high viscosity with forces artificially introduced by experimenters in a low viscosity regime.  This leads to a general formula \eqref{eq:high-viscosity-velocity} for calculating the velocity of a single motor with cargo as a function of solvent viscosity, and we provide a ``back-of-the-envelope'' formulation in equation \eqref{eq:velocity-approx}.  While these formulas provides a convenient interpolation between \emph{in vitro} and \emph{in vivo} velocities, our model does predict that the force required to stall a motor increases as the viscosity of the environment increases.  This counterintuitive result is not in agreement with the experimental work of Shubeita et al \cite{Shubeita:2008}.  Resolving this disagreement merits a more detailed analysis of the near-stalled motor mechanics, which we hope to address in future work. 

Having established these results for a single motor interacting with cargo, we proceed in Section \ref{sec:nocargofluc} to analyze the dynamics when two identical motors are attached to the cargo. There are two primary mechanisms by which one may expect multiple motors to enhance transport.  First, cargo bound to microtubules by multiple motors are able to stay attached, and hence stay in transport mode, longer.  Second, because the cargo load is distributed among the motors, each individual motor carries less load and should be able to process along the microtubule more quickly.  However, it has recently been observed that under light cargo loads, transport with multiple motors may be {\it slower} than with one \cite{kunwar2008stepping,Wang:2009,jamison:2010}.

We consider the latter phenomenon by way of two averaged models.  In Section~\ref{sec:nocargofluc}, we study an ``instant relaxation'' model where the cargo immediately moves to a position centered among the motors, a simplifying assumption also made in~\citet{Wang:2009}.  This is essentially equivalent to assuming the time scale of the cargo dynamics is fast compared to the motors (which we show in Section \ref{sec:nondim} is typically true \emph{in vitro}, but not necessarily \emph{in vivo}), 
and that thermal fluctuations of the cargo can be neglected (this latter assumption does not have a similar quantitative justification).  We compute transport properties of the motor-cargo system as a function of external force. In this simplified setting, we show that whether or not two motors are slower than one, in fact,  depends on the shape of the force-velocity curve through its influence on the spatial configurational dynamics of the motors.  This result is formalized in Theorem \ref{thm:qual-nf}. We also demonstrate that the force required to stall a two-motor system is more than twice that of a one-motor system.  This \emph{superadditivity} of stall force is in contrast to prior studies, suggesting that previously observed \emph{subadditivity} \cite{Kunwar:2010} may again be a consequence of attachment/detachment dynamics.

\section{The mathematical model}
\label{sec:discussion}

\subsection{Experimental Considerations}
\label{sec:experimental-considerations}

There are two primary experimental techniques for studying how teams of motors respond to a given load.  In microtubule gliding assays \cite{gagliano:2010}, the tails of kinesin, which are typically bound to cargo, are instead adhered to a glass plate. A microtubule is placed on the motor-coated plate; by binding to and taking steps on the microtubule, the kinesin motors generate sufficient force to transport the microtubule, whose average velocity is observed.  Optical tweezers \cite{Perkins:2009} are used to apply an opposing force to the trailing end of the microtubule while its resulting velocity is monitored. In a second type of experiment, optical tweezers are used to track latex beads, serving as \emph{in vitro} cargo, which bind to an unspecified number of motors and are transported along individual microtubules.  Through an active feedback system, the tweezers are capable of applying approximately constant force while tracking the beads.  We will  focus primarily on this second experimental setup.   

In nearly all of the previous theoretical efforts, experimental data for single motor-cargo systems has been used to build models for the behavior of motor-cargo systems with multiple motors.  In our model as well, experimental data for single-motor cargo systems plays a vital role. We briefly summarize the most important qualitative properties of single motor and cargo systems: The motors take discrete steps along the microtubule -- for kinesin-1, approximately 8 nm \cite{Coppin:1996}. The stepping rate depends on the local ATP concentration \cite{Visscher:1999} and on external forces applied to the cargo \cite{Visscher:1999,Carter:2005}. The coiled-coil tail between the motor heads and the cargo is a semi-flexible elastic structure \cite{Coppin:1996}. Eventually,  the motors will unbind from the microtubules, the cargo, or both. The stronger the opposing force applied to the cargo, the shorter the \emph{run length} before such detachment occurs \cite{Schnitzer:2000}.

In terms of the motion of motors along the microtubule, many existing models employ discrete state-space Markov chain methods ~\cite{Wang:2009,Campas:2006,Posta:2009,Goldman:2010,Driver:2010}. 
For example, in \cite{Wang:2009}, motors undergo random walks on a spatial lattice, with hopping rates biased according to the strain in their tails connecting to the cargo; the cargo, in turn, is assumed to instantaneously relax to a minimal energy position conditioned on the current position of the motors.   A lattice is well-suited for the discrete-state stepping process of the motors and lends itself to incorporating both the biochemical transitions and the continuous spatial dynamics of the unbound head and cargo. However,  the continuum model we choose is a natural coarse-graining obtained through systematic asymptotic reduction procedures~\cite{Wang:2007,Wang:2003,Fricks:2006,Lindner:2001,Reimann:2002,Frickshancock2,Frickshancock3,Pavliotis:2005}.   In particular, starting from a more detailed dynamical description for the molecular motor,
formulas for the drift and diffusion coefficients appearing in Eq.~(\ref{eq:motorcargosde}) can be computed through Markov chain calculations~\cite{Wang:2007,Wang:2003,Fricks:2006}, renewal theory~\cite{Lindner:2001,Reimann:2002,Frickshancock2,Frickshancock3}, or homogenization theory~\cite{Pavliotis:2005}.  

Rather than exploring the discrete-continuous connection in detail here, however, we build our model at the mesoscale, based on single-motor force-velocity and force-diffusivity relations. The qualitative features of these relations are informed by the robust literature of experimental measurements ~\citep{Visscher:1999,Kojima:1997,Shtridelman:2008,Shtridelman:2009}. Furthermore, while the models of~\cite{kunwar2008stepping,Korn:2009,jamison:2010,Kunwar:2010,Wang:2009} do have spatial detail and incorporate force-velocity relationships comparable to our model, our formulation permits a more flexible and faithful rendering of the random diffusive component of the motor dynamics. 
We find that this diffusive component, along with the shape of the force-velocity curve, plays a crucial role in Proposition \ref{thm:superadd-nf}, where we show that the model exhibits a superadditive growth of the stall force for a collection of motors.
Our choice of modeling framework renders asymptotic analysis possible; we identify which parameters are physically small and use stochastic averaging techniques  
to derive a simplified, lower-dimensional system of equations whose solution approximates the dynamics of the original model.  In particular, for the case of $ N=2 $ motors, we use this systematic dimension reduction to obtain a completely analytical relationship between the dynamics of the motor-cargo complex and the model for a single motor.

\subsection{An SDE modeling framework}

The location of a molecular motor along a microtubule as a function of time $ t $ is described by a one-dimensional spatial coordinate $ X_i(t) $, indexed by the motor  $ i  \in \{1,\ldots,N\} $.  In particular, we coarse-grain over details of the biochemical and conformational states of the motor, and we consider length scales (on the order of hundreds of nanometers) sufficiently large that the motor position can be reasonably described by a continuous spatial variable rather than in terms of discrete steps of a few nanometers.   Because we consider the case of $ N $ cooperative and identical motors, each motor has identical dynamics.   We represent the position of the cargo by a one-dimensional spatial coordinate $ Z(t) $ along the microtubule, neglecting transverse fluctuations.  Korn el al \cite{Korn:2009} numerically explore a similar spatial continuum model that allows a three-dimensional representation for the cargo, but since our purpose here is to apply stochastic averaging techniques to derive analytical coarse-grained descriptions, we choose to illustrate this process in the simplest case. 

In this model, the dynamics of each motor and the cargo to which it is bound are described by the following system of stochastic differential equations \begin{align}
	\difd X_i(t) &=  v g(\Fspring (X_i(t) - Z(t))/F_{*}) \, \difd t + \sigma \sigfun ( F (X_i(t)-\bZ(t))/F_{*}) \, \difd W_i(t)
	\nonumber \\
	\gamma \difd Z(t) &= \Big[\sum_{i=1}^{\Nmot} \Fspring (X_i (t),Z (t)) - \Ftrap\Big] \, \difd t + \sqrt{2 k_B T \gamma} \, \difd W_z(t),
	\label{eq:motorcargosde}
\end{align}
where  $ F_{*} $ is the stall force of a motor; $ v $ is the mean speed at which a motor moves along the track when no force is being applied; $ \frac{1}{2} \sigma^2 $ is the effective diffusivity of a motor along the microtubule when no force is applied; $ \Ftrap $ is the magnitude of any laser trap force applied to the cargo; $ T $ is the absolute temperature of the system; and $ \kB $ is Boltzmann's constant. 

For the friction constant of the cargo $\gamma$, we use the Stokes-Einstein drag law for spherical particles, $\gamma = 6 \pi a \eta$, where $a$ is the effective radius of the cargo and $\eta$ is the dynamic viscosity of the fluid environment. There are in fact several appropriate choices for $\gamma$ when one considers the various types and shapes of cargo found \emph{in vivo}, see Mitchell \& Lee \cite{Mitchell:2009} for a thorough discussion.  Our default simulations will be based on spherical 500 nm particles in a solvent with the viscosity of water at 300 K.  This is appropriate for describing a typical \emph{in vitro} experimental setup.   The cytoplasmic environment, however, is thought to have a viscosity as much as 600
 times that of water \cite{Mitchell:2009}, in Sections \ref{sec:onefluc} and \ref{sec:multimotor} we will study our model across a broad range of values of viscosity.  Rather than explicitly representing the results separately in terms of cargo size $ a $ and viscosity $ \eta $, which then would require a constant appeal to a spherical cargo assumption, we will present all results in terms of the more broadly applicable lumped friction constant $ \gamma $.

The nondimensional functions $g$ and $h$ are rescaled force-dependent drift and diffusion functions for the motors.  The drift determines the instantaneous expected velocity of the motor and the diffusion function determines the local diffusivity of the process.  Here we will simply take the diffusion as independent of the applied force ($ h \equiv 1 $), and 
discuss the drift function in greater detail in Section \ref{sec:g-and-h}.  Typical values for the dimensional constants appearing in (\ref{eq:motorcargosde}) are as listed in Table~\ref{tab:const} for kinesin-1~\cite{Muthukrishnan:2009p353,Kojima:1997,Visscher:1999}.

The functions $ \{W_i (t)\}_{i=1}^{\Nmot}, W_z (t) $ are independent standard Brownian motions: that is, each of $\{W_i(t)\}_{i=1}^{\Nmot}, W_z(t) $ is a continuous, Gaussian stochastic process with independent increments, zero mean, and variance equal to $t$; and $\{W_i(t)\}_{i=1}^{\Nmot}, W_z(t) $ are independent. There is a noteworthy distinction between modeling justification of the process $W_z (t)$ and the system of processes $\{W_i (t)\}_{i=1}^{\Nmot}$.  The SDE for $Z$ is the overdamped Langevin equation for a diffusing spherical particle, subject to external forces that appear in the drift term.  Were it not for these forces, $W_z(t)$ would represent the Brownian motion of a free particle.  By contrast, the processes $\{W_i (t)\}_{i=1}^{\Nmot}$ represent the deviations from the mean of the motor dynamics that arise from fluctuations in chemical reaction wait times (and other discrete random events) associated with stepping. For this reason, the SDE for the cargo should satisfy the fluctuation-dissipation relationship via the drag parameter $\gamma$, while the motor equations need not. 

The function $ F (\slen) $  describes the force exerted by the cargo on a motor through its connecting tail, in the direction against its natural forward motion, where $\slen$ is the signed separation between the motor and cargo positions.  As in the literature, we use spring-like models for this force law.  The simplest Hookean model 
\begin{equation}
\Fspring (\slen) = \kappa \slen \label{eq:lintail}
\end{equation}
with constant spring coefficient $ \kappa $, is actually fairly consistent with experimental data for kinesin~\cite{Coppin:1996,Kojima:1997}, provided the tail is not stretched too far ( $ \slen \lesssim 70 $  nm)~\cite{Hariharan:2009}, and we use it for clarity in our main development.   We stress that we consider the linear force law (\ref{eq:lintail}) merely as a convenient and reasonable phenomenological approximation; the connecting tail has an effectively jointed structure which gives rise to a more complex relationship between the force and the motor-cargo separation.
Our analysis can be readily extended to nonlinear spring models that have nonzero rest length \cite{kunwar2008stepping}, a sigmoid function as in \cite{jamison:2010} and the wormlike chain model~\cite{Howard:2001}, provided the force law is approximately linear over a wide enough range (\ref{sec:nonlinspring}).  %

\begin{table}
\begin{small}
\rowcolors{2}{lightgray}{}
\centering \begin{tabular}{r|c|l}
	\toprule
	\multicolumn{3}{c}{\textbf{Motor Properties}} \\
	\midrule[0.08em] Quantity & Label & Value for kinesin \emph{in vitro} \\ 
	\midrule[0.08em] Step size & $L$ & 8 nm\\ 
	Stall force & $F_*$ & 7 pN\\
	Unperturbed motor velocity  & $v$ & 500 nm/s \\
	Neck linkage spring constant 
	& $\kappa$ & 0.34 pN/nm
\\
	``Randomness Parameter'' & $r = \sigma^2 / L v$ & 1.2 \\
	Effective motor diffusion  & $\sigma^2$ & 5000 $\text{nm}^2/\text{s}$ \\ 
	\midrule[0.08em]\multicolumn{3}{c}{\textbf{Cargo and Environment Properties}} \\
	\midrule[0.08em] 
	Effective Cargo Radius & $a$ & 500 nm \\ 
	Fluid Viscosity & $\eta$ & $1 \times 10^{-9}$ $\text{pN} \cdot \text{s}/\text{nm}^2$ \\ 
	Optical trapping force & $\theta$ & $-10 $ to $ 20$ pN \\ 
	Boltzmann constant $\cdot$ temp  & $k_B T$ & 4.1 $\text{pN} \cdot \text{nm}$\\ 
	Cargo friction & $\gamma = 6 \pi a \eta$  & $1 \times 10^{-5}$ $\text{pN} \cdot \text{s} / \text{nm}$ \\  \bottomrule 
\end{tabular}
	\label{tab:const}
	\caption{Experimental estimates of physical parameters for kinesin \cite{Visscher:1999,shastry2010neck,Muthukrishnan:2009p353,Kojima:1997}. Our estimate for the cargo friction is based on the viscosity of water. In Section \ref{sec:onefluc} we assume that the \emph{in vivo} viscosity can be as high as $6 \times 10^{-7}$ $\text{pN} \cdot \text{s}/\text{nm}^2$, consistent with \cite{Mitchell:2009}.  
}
\end{small}
\end{table}

\subsection{Force-dependent instantaneous velocity} \label{sec:g-and-h}
We describe next the nondimensional drift functions $g$, which summarizes the detailed response of the motor's transport properties to an applied force.   The function $ g $ can be thought of as a nondimensionalized version of the force-velocity curves typically reported in experimental measurements of motors~\citep{Visscher:1999,Kojima:1997,Shtridelman:2008,Shtridelman:2009}, with the unforced motor velocity and stall force scaled so that $ g(0) = 1 $ and $ g(1) = 0 $.  Note that the argument $ f$ of $ g(f) $ is the ratio of the applied force, measured in the direction opposing natural motion of the motor, to the stall force.
In general, these functions $g$  decrease at an order unity rate for substall forces $ 0 \leq f \leq 1 $, and they saturate at finite constants both for superstall forces ($ f \geq 1 $) and forces applied along the direction of the motor's natural motion ($ f < 0 $).   Theoretical models~\citep{Muller:2008jstatphys,Korn:2009}  generally employ simple functions satisfying these conditions, though some models actually incorporate, for tractability rather than accuracy, functions $ g $ that grow linearly for large applied forces $f$ \cite{Muller:2008jstatphys}.   As a specific example that we employ in numerical simulations, we take a sigmoid function form for $g$, namely
\begin{equation}\label{E:gbar-form}
	g(f) = A - B \tanh (C f - D).
\end{equation}
We impose the constraints $g(0) = 1$ and $g(1) = 0$, as well as the asymptotic relations:
\begin{align*}
	\lim\limits_{f \rightarrow -\infty} g (f) & =\frac{v_{\max}}{v}, \qquad \lim\limits_{f \rightarrow \infty} g(f) =\frac{v_{\min}}{v},
\end{align*}
with $ v_{\max} $ the maximum speed the motor moves when pulled by a strong assisting force; $ v_{\min} $ the nonpositive velocity at which the motor moves backwards when dragged by a strong superstall force \cite{Carter:2005}; and $v$ the unperturbed motor velocity.
This uniquely determines the constants to be
\begin{equation*}
	A = \frac{v_{\max} + v_{\min}}{2v}, \, B = \frac{v_{\max} - v_{\min}}{2v}, \,  D = \text{arctanh}\left( \frac{2v - v_{\max} - v_{\min}}{v_{\max} - v_{\min}} \right)
\end{equation*}
\begin{equation*}
	\text{and } C = \text{arctanh}\left( \frac{v_{\max} + v_{\min}}{v_{\max} - v_{\min}}\right) - \text{arctanh}\left( \frac{v_{\max} + v_{\min}-2v}{v_{\max} - v_{\min}}\right).
\end{equation*}
Such a function with physically relevant parameters is depicted graphically in Figure~\ref{fig:gex}. 
\begin{figure}[ht]
	\centerline{\includegraphics[scale=0.3]{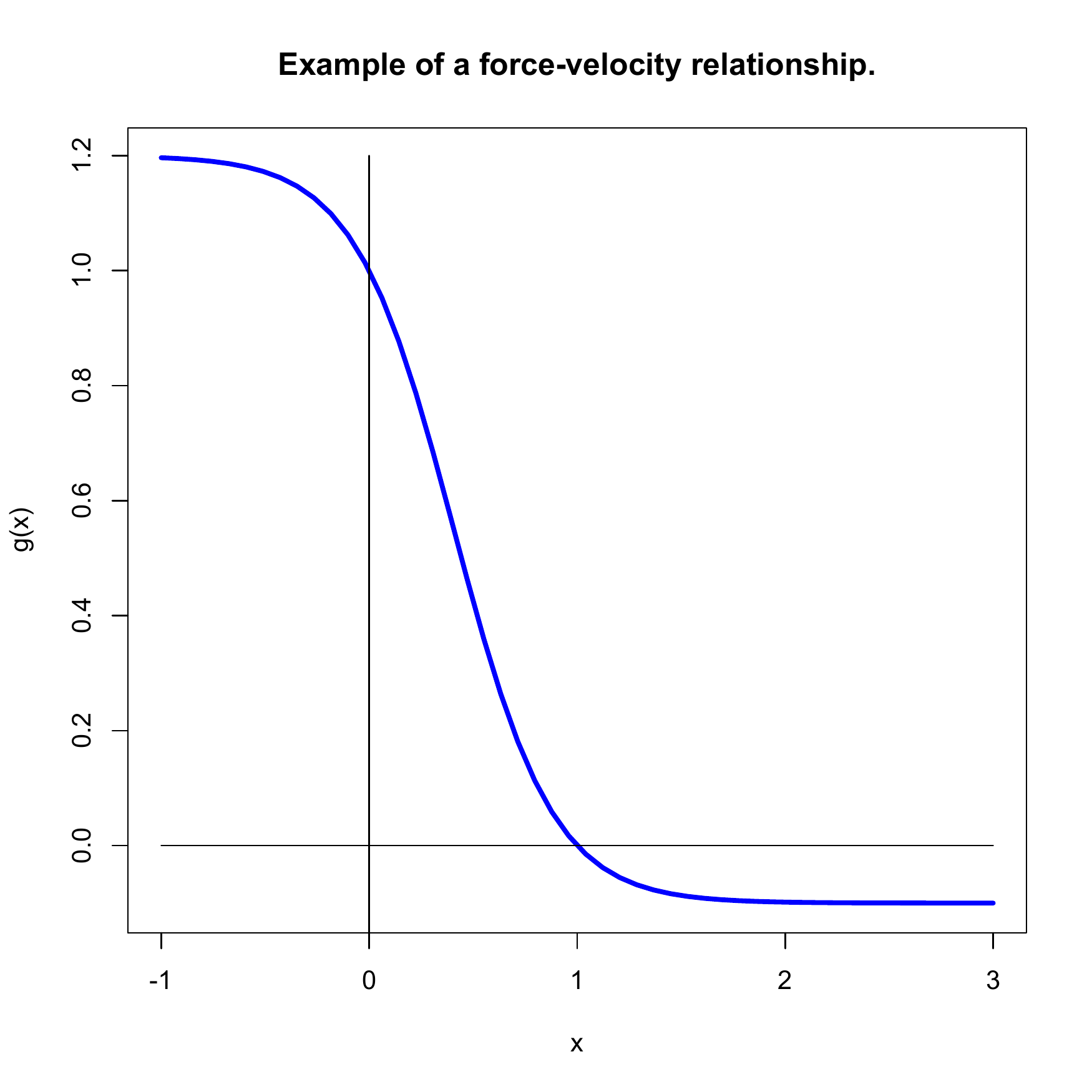}}
	\caption{An example model \eqref{E:gbar-form} for the force-velocity relation $ v g(f/F_*) $ with maximum velocity $v_{\text{max}} = 600$ nm/s, minimum velocity $v_{\text{min}} = -50$ nm/s, free velocity $v  = 500$ nm/s and stall force $ F_* = 7 pN$.   
\label{fig:gex}
}
\end{figure}
Although we have described a specific model for the force-velocity relationship, our methodology and conclusions apply to more general choices of models. As we assert in Theorems \ref{thm:qual-nf} and \ref{thm:qual}, the concavity of $g$ near zero force and near stall force is critical to qualitative behavior of a motor-cargo system when more than one motor is attached to the same cargo. (This is also noted by Wang and Li \cite{Wang:2009} and Kunwar and Mogilner \cite{Kunwar:2010}).  Our operating assumption is that the force-velocity curve is concave down for small external forces, and concave up for near-stall forces.  This corresponds to typical force-velocity curves for molecular motors under normal operating conditions~\cite{Visscher:1999,Carter:2005}, but we note that the force-velocity data from these sources also indicate a concave-up behavior for small external forces when the local ATP concentration is very low. (See \cite{Gross:2007,Mitchell:2009} for further discussion and contrasting properties for dynein). We summarize these qualitative properties, typical of the high ATP regime, as follows.

\begin{assumption} \label{a:g-qual}
	We assume that the instantaneous force-velocity relationship $g: \mathbb{R} \to \mathbb{R}$ is smooth, and bounded \cite{Carter:2005} with bounded derivatives.  We normalize $g$ with respect to the unencumbered velocity of the motor and the stall force such that
\begin{equation*}
	g(0) = 1, \text{ and } g(1) = 0.
\end{equation*}	
Furthermore we require $g$ to satisfy the following qualitative properties:
	\begin{enumerate}[(i)]
		\item (Monotonicity) $g(f)$ is strictly decreasing in $f$ \cite{Visscher:1999};
		\item (Concavity) There exists an $f_* \in (0, 1/2)$ such
		\begin{align*}
			g''(f) &< 0, \, \text{for all } f \in (-\infty, f_*]; \text{ and}\\
			g''(f) &> 0, \, \text{for all } f \in [1 - f_*, \infty);
		\end{align*}
		\item (Strong Concavity) \label{a:g-strong} 
		The function $\gsym(f) := g(f) + g(-f)$ is decreasing in $f$ for $f > 0$, while $\tilde \gsym(f) := g(1 + f) + g(1 - f)$ is increasing in positive $f$.
	\end{enumerate}
\end{assumption}

Note that functions of the form \eqref{E:gbar-form} satisfy Assumption \ref{a:g-qual}. Indeed, $g'(f) = - BC / \cosh^2(Cf - D)$ and so when $ |\vmin| < |\vmax| $ and $ v > (\vmax + \vmin)/2 $ (so that $A, B, C, D > 0$),
	\begin{align*}
		\frac{d}{df} \big(g(f) + g(-f)\big) =  -BC \left(\frac{1}{\cosh^2(Cf - D)} - \frac{1}{\cosh^2(-Cf - D)}\right) < 0
	\end{align*}
implying part the strong concavity property about 0. A similar calculation confirms the property with respect to $f = 1$.

We note though that the force-velocity relationship in the data collected by Carter and Cross \cite{Carter:2005} is not monotone for negative values of $f$. In fact, it is difficult to detect the precise properties in this regime because the motor is so apt to unbind from the microtubule. We choose to take from the data that the velocity of the motor is bounded, which is consistent with the observation that the stepping must be rate-limited by the binding rate of the forward head of the motor.

\section{Nondimensionalization and Identification of Distinguished Limits for Asymptotics}
\label{sec:nondim}

We begin with a consideration of the low viscosity environment characteristic of \emph{in vitro} experiments, for which the physical parameters in our model take values approximated by those presented in Table~\ref{tab:const}.
The equations (\ref{eq:motorcargosde}) can be readily simulated, but we exploit asymptotic techniques to reduce the coupled dynamical system to a more easily interpretable lower-dimensional system. 
To perform a nondimensionalization, we first choose suitable reference length and time scales to render the system (\ref{eq:motorcargosde}) nondimensional and to identify small parameters.  For this approach to produce physically meaningful small nondimensional parameters, the reference units must be chosen so that the nondimensional variables vary on order unity scales~\cite{Lin:1988,Novozhilov:1997}.   For a cooperative system of motors, the length and time scales can be chosen in the same way as for a single motor attached to the cargo, provided $ 1/\Nmot $ is not small compared to the resulting small nondimensional parameter.   This latter consideration is well-satisfied, since $ \Nmot \lesssim 10 $ and the small parameter that will emerge from the analysis has magnitude $ \epsilon \sim 10^{-3} $.   We remark that in other configurations, care is needed. In particular, in a tug-of-war configuration, the motors are all nearly stalled, and this will induce different length and time scales to their dynamics as compared to the fully cooperative configuration.   Moreover, while our nondimensionalization is still valid for the higher viscosities characteristic of \emph{in vivo} cellular environments, our conclusions about the relative sizes of the nondimensional parameters (particularly $ \epsilon $) will need to be revisited for this setting.

We consider  the relative size of the forces acting on the motor-cargo system, using the physical magnitudes of the various parameters as listed in Table~\ref{tab:const}, characteristic of an \emph{in vitro} experiment.
The frictional force $ \gamma v $ that would be exerted on the cargo if it were pulled at the full unencumbered speed $v $ of a motor is $ 2.5 \times 10^{-3} $ pN, by far the smallest force incorporated in the model.    
To characterize the strength of thermal effects, we suppose that the tail connecting the motor and cargo is in approximate thermal equilibrium, which implies (under the linear spring model) that its distortions impart a typical potential energy $ \oh \kB T $, so that 
$ \oh \kappa \E{(X (t)-Z(t))^2} \sim \oh \kB T $.
Consequently, the length of the spring has mean-square value 
$ \E{X(t) - Z(t))^2} \sim \kB T/\kappa, $
so the typical magnitude of force fluctuations induced for the connecting tail  is 
$$ \kappa \E{(X(t) - Z(t))^2}^{1/2} \sim \sqrt{\kB T \kappa} \sim 1.3 \text{pN}. $$    The stall force required to arrest the motor is, depending on the actual motor, in the range of 5 to 10 pN.  In experiments, the laser traps are generally used to apply forces that range from $-10$ to 20 pN.

These suggest the following leading order picture:   In the absence of a laser trap, the forces induced by thermal fluctuations are substantially stronger than those induced by friction, so the dynamics of the tail connecting the motor and cargo should be dominated by thermal fluctuations.  We therefore infer that a suitable reference length scale is $ \sqrt{\kB T/\kappa} \sim 3 \text{ nm}$, the typical length of the fluctuations in the tail due to thermal effects.   The connecting tail itself, in a realistic model, is actually on the order of $70$ nm, but as we discuss in \ref{sec:nonlinspring}, this length scale does not play any dynamical role, so we do not use it as a basis for our dimensional analysis.   Also, the $8$ nm step size of the molecular motor may appear relevant, but we are considering the effective behavior of the motor-cargo complex as it progresses over much larger distances; in our model this discrete-stepping detail has already been coarse-grained, with the effects incorporated into the force-velocity and force-diffusivity relations from Subsection~\ref{sec:g-and-h}.

To estimate the time scale, we note that neither the thermal forces nor the friction force are particularly strong compared to the stall force, so we crudely approximate the force-velocity relationship by its local linearization about zero force:
\begin{equation*}
g(f) \approx 1 + g^{\prime} (0) f,
\end{equation*}
where we have noted that $ g(0) = 1$.  Under this rough approximation, the dynamics of a cargo with a single motor can be written as the linear system:
\begin{equation}
\difd \left[\begin{matrix} X (t) \cr Z(t)\end{matrix}\right] = \Jmat \left[\begin{matrix} X(t) \cr Z(t)  \end{matrix}
\right] \, \difd t +
\left[\begin{matrix} v \cr -\Ftrap/\gamma \end{matrix}\right] \, \difd t +  \left[\begin{matrix} \sigma & 0 \cr 
0 & \sqrt{2 k_B T/ \gamma} \end{matrix}\right] \left[\begin{matrix} \difd W_x (t) \cr \difd W_z (t)
\end{matrix}\right]
\end{equation}
 with constant  matrix 
\begin{equation}
\Jmat = 
\left[\begin{matrix} - \frac{\kappa v}{\Fstall} g^{\prime} (0) & \frac{\kappa v}{\Fstall} g^{\prime} (0) \cr
\frac{\kappa}{\gamma} & - \frac{\kappa}{\gamma}\end{matrix}\right].
\end{equation}
This matrix has eigenvalues $ 0 $ (corresponding to no resistance when the motor and cargo move together)
and
\begin{equation}
\lambda = - \frac{\kappa v g^{\prime} (0)}{\Fstall} - \frac{\kappa}{\gamma},
\end{equation}
whose magnitude describes the effective rate constant of the tail stretch $ X(t) - Z(t) $.  The inverse of its magnitude provides a good reference time scale.  As noted above $ \Fstall \gg \gamma v $, so
$ \lambda \approx - \frac{\kappa}{\gamma} $ and we choose $ \gamma/\kappa $ as the reference time scale.  This is precisely the time scale which characterizes the dynamics of the tail when attached to the cargo and a fixed object, which is appropriate, since the motor dynamics are relatively slow compared to the cargo, a fact we will see in the sequel.  This choice of reference time scale is also valid for a wide range of laser trap forces and for the case of multiple cooperative motors.  It is true that for laser trap forces large enough to bring the motors near to a stall, the length and time scales change, but  our aim is to select reference length and time scales that work well for a range of laser trap forces, from zero to stall.  By neglecting the laser trap force, we can specify length and time scales that are relevant for a much wider range of forces than those based on stalled configuration dynamics.

While mass units technically enter an equation involving force balance, none of the variables of interest involve mass units and we can simply divide the equation for $ Z(t) $ in Eq.~\ref{eq:motorcargosde} by $ \gamma $ so that 
mass units disappear from the left hand side and therefore also from every group of parameters on the right hand side.  Consequently, we effect our nondimensionalization by combining this simple manipulation with the explicit rescaling of variables with respect to length scale $ \sqrt{2\kB T/\kappa} $ (the factor of $ 2 $ for convenience) and time scale $ \gamma/\kappa $:
\begin{align*}
\Xnd (\tt) = \frac{X (\gamma \tt /\kappa)}{\sqrt{2\kB T/\kappa}}, \qquad \Znd (\tt) = \frac{Z (\gamma \tt /\kappa)}{\sqrt{2\kB T/\kappa}},
\end{align*}
where $\tt=(\kappa/\gamma) t$.  We also employ an important fact about Brownian motion: that scaling time, $W(a t)$, is equivalent in distribution to scaling space, $\sqrt{a}W(t)$, for any positive $a$. The governing equations in terms of these nondimensional variables therefore become
\begin{align}
\label{eq:motorcargorescsde}\difd \Xnd_i(\tt) &=  \epsilon g( s (\Xnd_i(\tt) - \Znd (\tt))) \, \difd \tt + \sigmc \, \difd W_i(\tt)\\
\nonumber	\difd \Znd (\tt) &= \sum_{i=1}^{\Nmot} \left(\Xnd_i (\tt) - \Znd (\tt) - \frac{\Ftrnd}{\Nmot}\right) \, \difd \tt  +  \difd W_z(\tt).
\end{align}
where we have identified the nondimensional parameters listed in Table~\ref{table:nondim}.  
\begin{table}
\begin{small}
\begin{center}	\begin{tabular}{l|c|c|c}
	\toprule \multicolumn{4}{c}{\textbf{Dimensionless Groups}} \\
	\midrule[0.08em]
		\hspace{0.7 in} Description & Label & Definition & Value \emph{in vitro} \\
	\midrule[0.08em]
	\begin{footnotesize}\begin{tabular}{l} Ratio of the friction force exerted by the \\ cargo when pulled by the natural motor \\ speed to the typical force exerted by the \\ tail due to thermal fluctuations. \end{tabular} \end{footnotesize} & \begin{large}$ \epsilon $ \end{large} & $\displaystyle\frac{v \gamma}{\sqrt{2k_BT\kappa}} $ & $3 \times 10^{-3}$ \\
	\midrule
	\begin{footnotesize}\begin{tabular}{l} \emph{Stallibility}. The ratio of the typical tail  \\ force to the stall force.  \end{tabular} \end{footnotesize} &
	\begin{large} $s$ \end{large} & $\displaystyle \frac{\sqrt{2k_BT\kappa}}{F_{\ast}} $ & 0.2 \\
	\midrule 
	\begin{footnotesize}\begin{tabular}{l} The ratio of the force applied by the laser \\ trap to the typical tail force. \end{tabular} \end{footnotesize} &
	\begin{large}$ \tilde{\theta}$ \end{large} & $\displaystyle \frac{\theta}{\sqrt{2 k_B T \kappa}} $ &  $-6$ to $12$ \\
	\midrule 
	\begin{footnotesize}\begin{tabular}{l} The ratio of the effective diffusion coefficient \\ of the motors to the standard Stokes-\\Einstein diffusion coefficient for the cargo. \end{tabular} \end{footnotesize} &
	\begin{large}$ \sigmc^2 $\end{large}& $\displaystyle 
	\frac{ \sigma^2}{2 k_B T/\gamma} $ & $6 \times 10^{-3}$\\
	\midrule 
	\begin{footnotesize}\begin{tabular}{l} Randomness parameter~\citep{Wang:2007b,Fisher:2001} \\ or inverse P\'{e}clet number~\citep{Kramer:2010} \\ with respect to length scale of \\ thermal fluctuations in tail.  \end{tabular} \end{footnotesize} &
	\begin{large}$ \rho$ \end{large} & $\displaystyle \frac{\sigma^2\sqrt{\kappa}}{v\sqrt{2 k_B T}} $ &  2 \\
	\bottomrule
	\end{tabular}
\end{center}
\end{small}
\caption{\label{table:nondim} Nondimensional groups with the range of values associated to the constants in Table~\ref{tab:const} with viscosity corresponding to that of water (representing an \emph{in vitro} environment).
}
\end{table}

We find that, \emph{in vitro}, the smallest of the nondimensional parameters is $ \epsilon \sim 3 \times 10^{-3}$. The diffusion coefficient ratio $ \sigmc \sim 8 \times 10^{-2}$ is also small relative to the parameters $ s $ and $ \Ftrnd $. 
This suggests the biophysical relevance of the following distinguished asymptotic limit, which we will see also has mathematically desirable properties: $ \epsilon $ is taken as the fundamental small nondimensional parameter, and we rewrite $ \sigmc $ as a parameter proportional to $ \sqrt{\epsilon} $, i.e., $ \sigmc = \sqrt{\epsilon \rho} $, where $ \rho $ is defined in Table~\ref{table:nondim}.
The nondimensional parameters $ s $, $ \tilde{\theta} $, and $ \rho $ will be treated as order unity and independent of the small parameter $ \epsilon $.  

As we have noted, the viscosity of the fluid environment ranges widely from the experimental setting to what is expected in the cytoplasm \cite{Gross:2007,Mitchell:2009}. We therefore treat $ \gamma $ as a flexible parameter, and our distinguished asymptotic  limit described above remains relevant as $ \gamma $ is varied (noting how it appears in the nondimensional groups in Table~\ref{table:nondim}), provided the parameter $\epsilon$ remains small, which we shall  see in Section~\ref{sec:low-viscosity}
means no greater than $O(10^{-1})$. If so, we call the dynamics \emph{diffusion dominated}; otherwise for reasons that become clear in Section \ref{sec:onefluc}, we call the dynamics \emph{drag dominated}. 

With the distinguished limit described above, the small parameter $\epsilon $ appears in the governing equations
in two places in the equation for the motor variables $ \Xnd_i (t) $.  Because of the appearance of $ \epsilon $ in the drift term and $ \sqrt{\epsilon} $ in the diffusion term, we can interpet $ \epsilon^{-1} $ as a ``slow'' nondimensional time scale characterizing the motor dynamics relative to the order 1 ``fast'' time scale characterizing the cargo (or equivalently, the connecting tail).   The physical justification for this time scale separation is that the cargo friction is very small relative to the spring force, and therefore the cargo (and spring) equilibrate quickly relative to the motor's motion along the track.  The relevance of thermal fluctuations, however, implies that this equilibration is not to a rest point, but rather to a stationary probability distribution of fluctuations about the minimal energy point of the spring. More specifically, consider the time change $\tt=\epsilon^{-1}\tavg$. 
Application of stochastic averaging theory (see \cite{Freidlin:1998,Skorokhod:2002,Khasminskii:1968}) allows the derivation of effective equations for the motor coordinates, without explicit reference to the cargo position (see~\cite{Peskin:1994,Peskin:1995,Elston:2000}
for analogous reductions for single-motor case).
When $ \epsilon $ is small, the behavior of the motor dynamics $  \{\tilde{X}_i (\epsilon^{-1} \tavg)\}_{i=1}^n $
is well-approximated, in the sense of weak convergence, by the solutions $ \{\Xavg_i (\tavg)\}_{i=1}^n $ of the following SDE,
\begin{equation}
\difd \Xavg_i(\tavg)=\gavg_i(\{\Xavg_i\}_{i=1}^{\Nmot}(\tavg))\, \difd \tavg + \sqrt{\rho} \difd W_i(\tavg), \hspace{0.1in} i=1, \cdots, \Nmot
\label{eq:motavg}
\end{equation}
where
\begin{align}
\gavg_i(\vec{x};\Ftrnd) &= \int_{\mathbb{R}} g\left(s(x_i-z)\right) \pi_Z(\vec{x},z;\Ftrnd) \, \difd z; \label{eq:g-avg}\\ 
\pi_Z(z;\vec{x},\Ftrnd) &= \sqrt{\frac{\Nmot}{\pi}} \exp\left[-\Nmot \left(z-\left[\frac{\sum_1^\Nmot x_i}{\Nmot} - \frac{\Ftrnd}{\Nmot}\right]\right)^2\right].
\label{eq:cargo-distr}
\end{align}
The function $ \pi_Z$ is the density of the stationary distribution of the cargo position, given fixed motor positions $ \vec{x} = \{x_1,\ldots,x_{\Nmot}\}$.  

Note that while the cargo variable has been removed from explicit consideration, the effective dynamics of the motors are coupled together because they are all connected to the cargo.  That is, the effective drift coefficient 
$ \gavg_i  $ for motor $ i $ depends on the current positions of not only motor $i $ but the other motors as well.  By following the changes of variables we have performed in this section, we conclude that the simplified equation \eqref{eq:motavg} provides a good approximation (provided the parameter $ \epsilon $ is sufficiently small) to the original model equations \eqref{eq:motorcargosde} in the sense that the statistical dynamics of the solutions $ \{X_i (t)\}_{i=1}^{\Nmot} $ to the true model equations \eqref{eq:motorcargosde} are well approximated by the statistical dynamics of $ \sqrt{2 \kB T/\kappa} \Xavg ( \tavg )$ evaluated at the rescaled nondimensional time $ \tavg = v  \sqrt{\kappa/(2 \kB T)}t$,
where
$ \{\Xavg_i (\tavg)\}_{i=1}^{\Nmot} $ are solutions to the simplified equations \eqref{eq:motavg}.  We will explicitly check the quality of this approximation in Section~\ref{sec:onefluc}.

\section{Analysis of transport by one motor.}
\label{sec:onefluc}

Numerous authors have expressed concern about translating experimental observation into \emph{in vivo} predictions. Shubeita et al \cite{Shubeita:2008} and Mitchell \& Lee \cite{Mitchell:2009} provide excellent discussions on the topic. At issue is that the now classical force-velocity relationships \cite{Visscher:1999,Carter:2005} were measured in the presence of a water-like fluid environment, while the viscosity \emph{in vivo} may be 100-1000 times greater.  To infer the effect of high viscosity, Shubeita et al \cite{Shubeita:2008} attempted to measure the stall force by trapping lipid droplets that are driven by kinesin-1 motors.  While they measured the low viscosity stall force to be near 5 pN, the apparent stall force \emph{in vivo} was closer to 2.4 pN.  In a more theoretical treatment, Mitchell and Lee \cite{Mitchell:2009} used the Stokes-Einstein relationship to estimate the drag force incurred on a motor via various shapes and types of cargo.  Their goal was to estimate whether a single motor could overcome the anticipated drag force in order to produce experimentally observed motor speed. 

As was indicated in the previous section, we find that there is a qualitative change between the low and high viscosity regimes.  We will show that nevertheless our modeling framework can provide a smooth interpolation
between them to explain how the presence and dynamics of the cargo affects the speed and effective stall force for a single motor.

\subsection{Low viscosity regime and inference for $g$}
\label{sec:low-viscosity}

When the viscosity of the fluid environment is similar to that of water, we can compute the mean velocity and diffusivity of the system by averaging the force-velocity curve $g$ against the quasi-stationary distribution of the cargo,  $\pi_{Z}$, with the motor position essentially held constant. 
For the linear spring model we are using in the main text, the stationary distribution of the cargo is Gaussian~\eqref{eq:cargo-distr}, and the average velocity and effective diffusivity for one motor with cargo can be expressed as
\begin{align}
	\nonumber \vf{1}{\Ftrnd} &= \lim_{\tavg \to \infty} \frac{\Xavg(\tavg)}{\tavg} = \gavg(\Ftrnd), \\
	\label{eq:one-motor-diffusivity} \df{1}{\Ftrnd} &= \lim_{t \to \infty} \frac{1}{2\tavg} \left(\Xavg(\tavg) - \vf{1}{\Ftrnd}) \tavg \right)^2 = \frac{\rho}{2},
\end{align}
where we have introduced the Gaussian average:
\begin{equation} \label{eq:one-motor-guass-avg}
	\gavg(\Ftrnd) = \int_{\rbb} g(s(x - z)) \frac{1}{\sqrt{\pi}}e^{-(z - (x - \Ftrnd))^2} dz = \gaussbigbig{s\Ftrnd}{\frac{s^2}{2}}{g}.
\end{equation}
The last equality follows by introducing the variable $y = s(x-z)$.  As expected, this effective drift $ \vf{1}{\Ftrnd} $ is not dependent on the location of the motor, only its averaged distance from the cargo. 

Written in terms of the original parameters, we have
\begin{equation} \label{eq:velocity-one-motor}
	\lim_{t \to \infty} \frac{X(t)}{t} \approx \vforig{1}{\theta} := v \gaussbigbig{\frac{\theta}{F_*}}{\frac{k_B T \kappa}{F_*^2}}{g}.
\end{equation}
\begin{floatingfigure}[r]{0.48\textwidth}
\centering
\includegraphics[scale=.38]{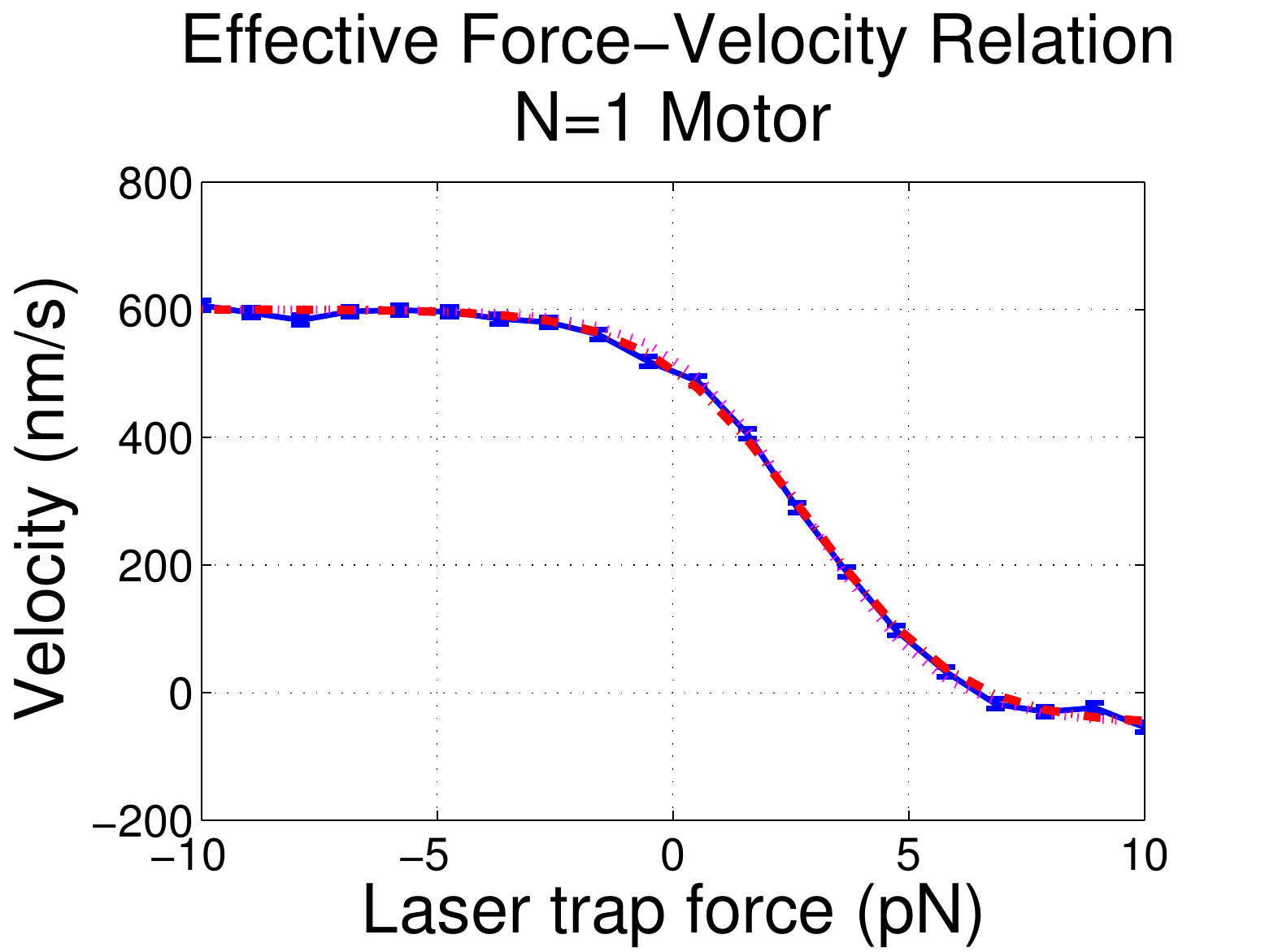} 	\caption{\label{fig:oneforcevelocity} Comparison of the effective force-velocity relationship of a single motor attached to a cargo from direct stochastic simulations of the model equations (\ref{eq:motorcargosde}) (blue with error bars denoting sample standard deviation) and the stochastic averaging result (\ref{eq:one-motor-guass-avg}) (red solid).   We also show (magenta dashed) the instantaneous force-velocity curve for a motor $ v g (\theta/F_*) $ which neglects thermal fluctuations.  The curves are somewhat difficult to distinguish as they are almost coincident.  Parameter values are listed in Table~\ref{tab:const}.
}
\end{floatingfigure}
Since $ \frac{k_B T \kappa}{F_*^2} = 0.03 \ll 1 $, it follows that for a given external force $\theta$, $\vforig{1}{\theta} \approx v g(\theta / F_*)$, which
corresponds to the naive estimate that the motor should move overall at a speed corresponding to the applied force $ \theta $ and the force-velocity relationship $ g $.

Furthermore, this characterization of the average velocity gives a clean representation of the effects of perturbing the physical parameters of the system. For example, the velocity given by \eqref{eq:velocity-one-motor} does not depend on the viscosity of the fluid environment, consistent with observations \emph{in vitro} \cite{Gross:2007} and \emph{in silico} \cite{Kunwar:2010}. The stability of the velocity with respect to viscosity encourages the notion the observed force-velocity curve should be robust as long the assumption that the ATP concentration remains constant holds.

The quantity \eqref{eq:velocity-one-motor} is the experimentally observed velocity, so the theoretical instantaneous force-velocity curve $g$ is the function that best fits $V_1(\theta)$ after ``inverting'' the Gaussian average.  Such inference can be done using the recently developed non-parametric statistical methods of Asencio et al \cite{asenciofunctional}. Figure~\ref{fig:oneforcevelocity} compares the instantaneous force-velocity curve for a motor with the ``dressed'' version that accounts for the effects of thermal fluctuations in the cargo (and therefore connecting tail).

A second example of a prediction made by \eqref{eq:velocity-one-motor} is found by perturbing the spring constant $\kappa$. For small $\theta$, we note that the velocity decreases when the spring constant increases.  This is due to the downward concavity near zero and the fact that increasing $\kappa$ widens the Gaussian average.

One final novelty of the formula \eqref{eq:velocity-one-motor} is that if $g$ were concave up near 0, as is the case in the low ATP regime, the opposite inequality would hold, meaning that the fluctuations would improve motor progress.  This can be interpreted as an example of the phenomenon called \emph{strain gating} \cite{kunwar2008stepping}. It has been observed that when force is applied to the motors, they can take steps even when little or no ATP is available. In the presence of cargo fluctuations, the cargo will explore regions of space that aid the motor in taking steps.  Because 
the motor would respond more to pulls from the cargo when it fluctuates in the forward rather than the backward direction, the average velocity of the system is greater in the forward direction than it would be in absence of cargo fluctuations.

All of these properties hold as long as the time scale separation holds between the motor and cargo dynamics, i.e., as long as $\ep$ is relatively small. In Figure~\ref{fig:oneforcevelocity}, we present the average velocity of the original system as a function of the cargo friction, which is proportional to the solvent viscosity. 
In agreement with a similar plot in \cite{Kunwar:2010}, the velocity holds approximately constant for two orders of magnitude greater than that of water, or in our nondimensional terms, for $ \epsilon \lesssim 0.1$.  Beyond that we must use a different analysis, which is provided in the next section.

\subsection{Crossover to high viscosity regime}
\label{sec:high-viscosity}

Because the stochastic averaging no longer holds when the solvent viscosity (or bead size) is so large that $\ep \sim 1$ or greater, we return to the system \eqref{eq:motorcargorescsde}.  In the one motor case, there is a simplification that arises from introducing $Y = X - Z$, the time-dependent distance along the microtubule between the motor and cargo. After a nondimensionalization and under the timescale $\tavg = \ep^{-1} \tnd$, we obtain
\begin{align}
\nonumber \difd \Xna(\tavg) &=  g( s \Yna(\tavg)) \, \difd \tavg + \sqrt{\rho} \, \difd W_x(\tavg)\\
	\difd \Yna (\tavg) &= [g(s \Yna(\tavg)) - \frac{1}{\ep}(\Yna (\tavg) - \Ftrnd )] \, \difd \tavg  +  \sqrt{\rho + 1/\ep} \, \difd W_y(\tavg).
\label{eq:motorcargorescsde-nd}
\end{align}
where $W_y = \frac{1}{\sqrt{\rho + 1/{\epsilon}}}(\sqrt{\rho}W_x - \sqrt{1/{\epsilon}} W_z)$ is a standard Brownian motion (correlated with $ W_x $). We see that the drift term of $\Yna$ is a function of  $\Yna$ alone, and the diffusion term is simply a scaled Brownian motion. Now, as a one-dimensional SDE, the drift term for $\Yna$ can be written as the derivative of a potential function, and we can solve for the stationary distribution of $ \Yna $ as an exponential of this potential:
\begin{equation} \label{eq:defn-pi-y}
	\pi_Y(y) = C\exp\left[\frac{1}{1 + \ep \rho} \left(-(y-\Ftrnd)^2 + 2\ep \int_0^{y} g(sy') \, \difd y' \right)\right].
\end{equation}
This formula reveals the manner in which the high viscosity regime is a perturbation of the low viscosity regime: when $\ep$ is small, $\pi_Y$ is essentially a Gaussian density with mean $\Ftrnd$. As before, to find the long-term velocity we average the drift term for $\Xnd$ against this stationary density,
\begin{equation} \label{eq:high-viscosity-velocity}
	\lim_{t \to \infty} \frac{X(t)}{t} = \lim_{\tilde t \to \infty} v\frac{\Xna(\tavg)}{\tavg} = v \int_{\rbb} g(sy) \pi_Y(y) dy.
\end{equation}
This expression can be rewritten more suggestively as:
\begin{align}
\nonumber \lim_{t \to \infty} \frac{X(t)}{t} 
&= v \int_{\rbb} \frac{1 + \ep \rho}{2 \ep} \left[\frac{\partial \pi_Y (y)}{\partial y} + \frac{2 (y - \Ftrnd)}{1 + \ep \rho} \pi_Y (y)\right] \, \difd y \\ 
&= \frac{v}{\ep} \left( \int_{\rbb} (y - \Ftrnd) \pi_Y (y) \, \difd y\right) = \frac{v (\langle \Yna \rangle - \Ftrnd)}{\ep}
= \frac{\kappa  \langle Y \rangle - \theta}{\gamma} \label{eq:convertvep}
\end{align}
where $ \langle \cdot \rangle $ denotes a statistical average over the fluctuations in the connecting tail length (governed by $ \pi_Y (y) $).  The final expression indicates that the speed of the system is given by the average net force on the cargo (connecting tail force minus applied load) divided by its friction coefficient.  Note that the nondimensional expressions approach $ 0/0 $ in the $ \ep \rightarrow  0 $ limit, which is why we did not recast our stochastic averaging expression \eqref{eq:velocity-one-motor} in this form.

To gain some further intuition, it is worth considering a deterministic version of the system \eqref{eq:motorcargorescsde-nd} defined by the system of ODEs,
\begin{align*}
	\ddtavg \, \Xndd = g(s \Yndd), \qquad 
	\ddtavg \, \Yndd = g(s \Yndd) - \frac{1}{\ep} (\Yndd - \Ftrnd).
\end{align*}
To find the average velocity, we find the stable equilibrium for $\Yna(t)$, namely a value $y_*$ such that $g(sy_*) = \frac{1}{\ep}(y_* - \Ftrnd)$.
Because $g$ is decreasing in $y_*$, the right-hand side is of $ \ddtavg \, \Yndd $ is strictly decreasing, so we know that if an equilibrium exists, this point is unique.  For the moment, assume that $g$ is linear for $y \in [0,1]$ and therefore of the form $g(y) = 1 - y$ to maintain consistency with Assumption \ref{a:g-qual}.  Then the above equation becomes $1 - sy_* = (y_* - \Ftrnd)/\ep$
and solving for $y_*$ gives $y_* = (\ep + \Ftrnd)/(\ep s + 1)$.

This implies that, for $\Ftrnd \in [0,1]$, the asymptotic nondimensional velocity for linear $g$ satisfies
\begin{equation*}
	\ddtavg \, \Xndd = g(s y_*) = \frac{1}{\ep} (y_* - \Ftrnd) = \frac{1 - s \Ftrnd}{1 + s \ep}.
\end{equation*}
Note the similarity to the penultimate expression in Eq.~(\ref{eq:convertvep}), with the statistical average over fluctuations now replaced by a deterministic value for the stable length $ y_* $ of the connecting tail.
In the original parameters of the problem, we have the approximation
\begin{equation} \label{eq:velocity-approx}
	V_1(\theta) \approx \frac{v(1 - \theta/F_*)}{1 + \gamma v/F_*} \, \text{ for } \theta \in [-F_*,F_*].
\end{equation}
In fact, this approximate formula remains unchanged even in the presence of thermal fluctuations ($\rho > 0 $), as can be shown by by simply substituting $ g(y) = 1- y $ into \eqref{eq:high-viscosity-velocity}. 
Figure~\ref{fig:onehigviscosity} shows that this simple explicit formula accurately represents the effective behavior of a single motor even into the high viscosity regime (beyond the validity of the stochastic averaging formula \eqref{eq:one-motor-diffusivity}), for applied forces up to 5 pN in magnitude. 

\begin{figure}[ht]
\begin{center}
\includegraphics[scale = 0.4]{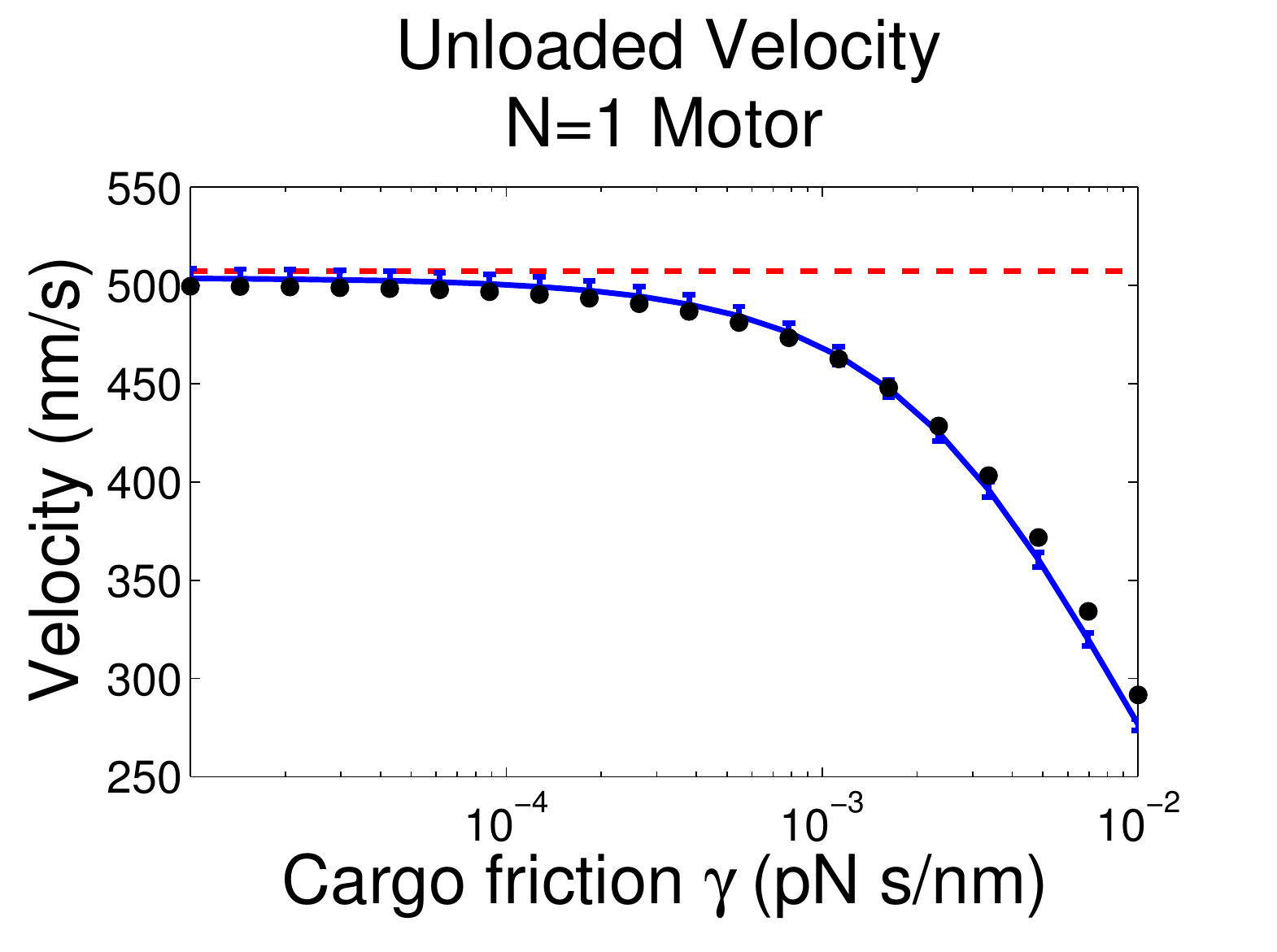}
\includegraphics[scale=.4]{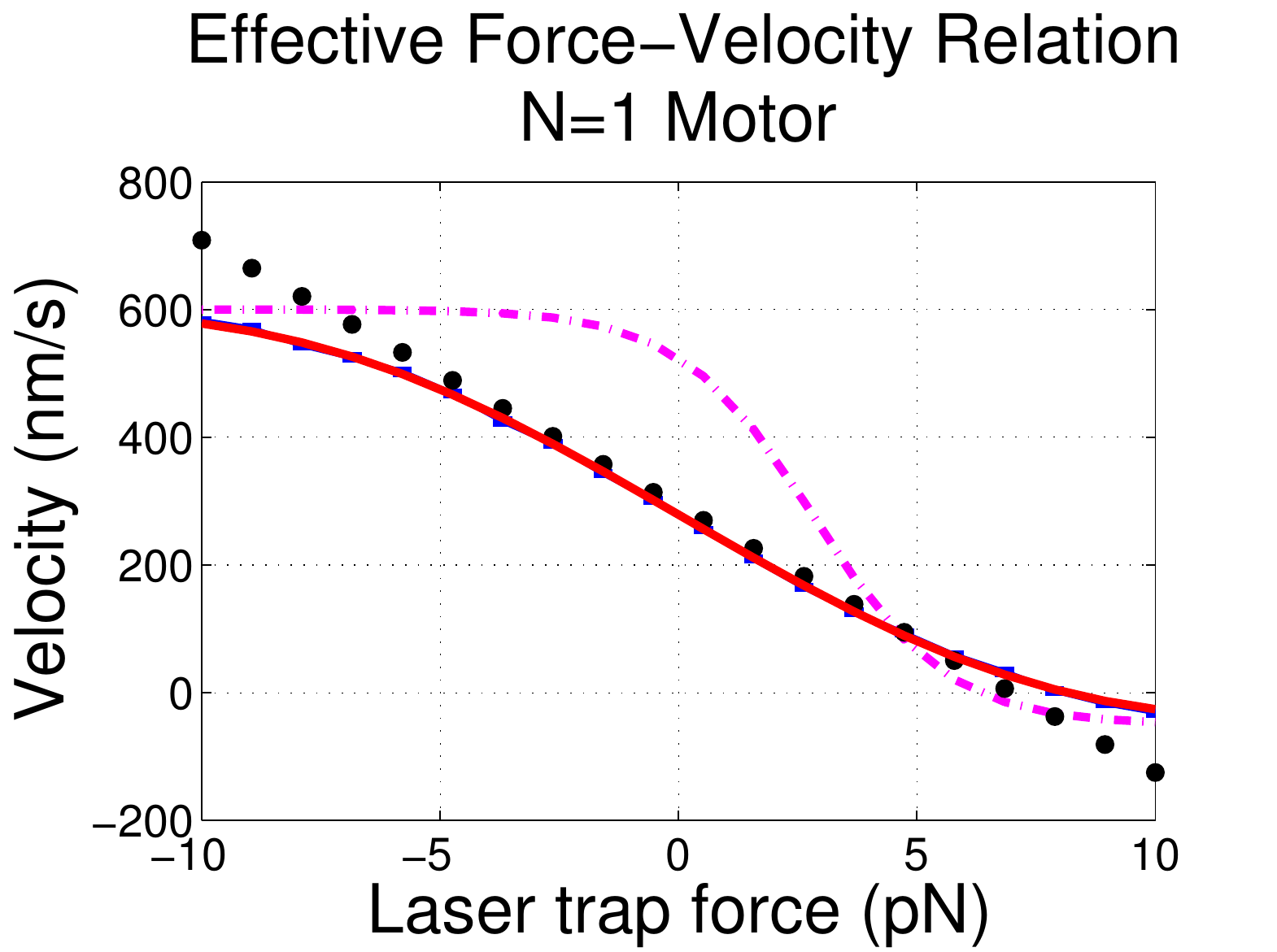}     \caption{\label{fig:onehigviscosity} Performance of the linear approximation. Left panel: Velocity of cargo attached to single motor as function of cargo friction parameter $ \gamma $.  Results of direct simulation of (\ref{eq:motorcargosde}) are indicated by the blue curve, with error bars representing one sample standard deviation.  The black dotted curve represents the linear approximation formula (\ref{eq:velocity-approx}), whereas the red dashed line indicates the theoretical value corresponding to the $ \epsilon \downarrow 0 $ limit.   The values of $ \gamma $ presented correspond to a bead of size $ a = 500 nm $~\citep{guydosh2006backsteps}
and solvent viscosity ranging from the viscosity of water (at 300 K) to 1000 times that value. Right panel:
Comparison of the effective force-velocity relationship of a single motor attached to a cargo with friction $ \gamma = 10^{-2} \mathrm{pN} \cdot \mathrm{s}/\mathrm{pN} $ ($ \epsilon = 3 $)  from direct stochastic simulations of the model equations (\ref{eq:motorcargosde}) (blue with error bars denoting sample standard deviation), the precise theoretical formula (\ref{eq:high-viscosity-velocity}) with the high viscosity modification (red solid), and the explicit formula (\ref{eq:velocity-approx}) for high viscosity modification based on a linear approximation (black dotted).
  We also show (magenta dashed) the instantaneous force-velocity curve for a motor $ v g (\theta/F_*) $ which neglects thermal fluctuations.    Other than the cargo friction $ \gamma $ (and proportionally related viscosity $ \eta $), parameter values are listed in Table~\ref{tab:const}.
}
\end{center}
\end{figure}

In particular, this formula gives a simple representation for how the velocity of a single motor with zero external force should decrease as a function of viscosity. We immediately read off a critical nondimensional ratio, $\gamma v / F_*$, which compares the motor's stall force to the drag force that would be felt by the cargo if the motor could achieve its natural low viscosity speed.  As can be noted from Figure \ref{fig:onehigviscosity}, the changeover from this ratio being negligible to significant occurs near the \emph{in vivo} viscosity estimated by Mitchell \& Lee \cite{Mitchell:2009}.  

The formula \eqref{eq:velocity-approx} does not do well, however, at large applied force; it predicts the stall force to
remain exactly $F_*$, regardless of the viscosity of the fluid.  
As we will show in Section \ref{sec:compare-no-fluc}, the stall force of the full stochastic model in \eqref{eq:high-viscosity-velocity} remarkably increases with the viscosity of the fluid due to the upward concavity of $g$ near stall force.  
This claim is counter to the observations of Shubeita et al \cite{Shubeita:2008}, for example.  A deeper version of the model and further analysis are necessary for making predictions in the near-stall regime.

\section{Analysis of transport by multiple identical motors}
\label{sec:multimotor}

Experimentally it is very difficult to determine precisely how many motors are attached to an observed cargo at a given time.  Typically, in order to get some certainty, transport is observed in a fluid medium that has a very low concentration of kinesin. A majority of the observed cargos are immobile (or are undergoing simple diffusion) but a minority are seen to move almost deterministically for a time.  Because the kinesin are so rare in the solution, it is hypothesized that only a single kinesin is bound to the cargo and a microtubule.  When the concentration of motor proteins is increased, not only is the mobile portion of the cargo population increased, but the run length and response to external forcing is changed.  In particular, motor concentration is expected to be often large enough \emph{in vivo} for multiple motors to bind to a cargo.   This motivates the development of quantitative relationships between \emph{in vitro} observations involving a single motor bound to the cargo to \emph{in vivo} predictions in which multi-motor transport is presumed to be relevant.  The purpose of most of this section is to derive such analytical connections within the modeling framework laid out in Section~\ref{sec:discussion}.

A striking observation from theorists and experimentalists is that the attachment of multiple motors to a given cargo does not necessarily imply a larger average velocity of the motor-cargo system. This has been observed separately by Kunwar et al \cite{kunwar2008stepping}, Wang and Li \cite{Wang:2009} and Bouzat and Falo \cite{Bouzat:2010}.  Interestingly, each group gives a distinct reason for the phenomenon: Kunwar et al cite detachment and attachment dynamics; Bouzat and Falo cite excluded volume effects among the motors; and Wang and Li cite the structure of the force-velocity curve.  We highlight the multiplicity of explanations in order to emphasize the richness of the underlying biophysics and the benefit of careful analysis. Our model also reproduces the two-slower-than-one result in the low viscosity environment, for reasons closest to that of Wang and Li. However, as we see in Section \ref{sec:twomotor-highvisc}, our model predicts that this phenomenon disappears when viscosity is closer to that expected in the cytoplasm.

\subsection{Analysis for a semi-deterministic approximation}
\label{sec:nocargofluc}
The two sources of stochasticity, cargo fluctuations and variation in the spatial configuration of the motors, both contribute in a non-trivial way to two-motor transport.  When both are eliminated, we have the model for force distribution introduced by M\"uller, Klumpp and Lipowsky \cite{Muller:2008pnas,Muller:2008pnas,Muller:2010}.  Each motor is presumed to bear half the load of the cargo and the resultant drift term in the original parameters of the system is $v g(\theta / 2 F_*)$.  We will refer to this as the \emph{force-balance theory}.

In order to distinguish between the contributions of the two sources of randomness, we add one back and then the other.  Because it turns out to account for a dominant contribution of new behavior, we first analyze the system with changing spatial configurations in the absence of cargo fluctuations. Versions of such an approximation, which is not a physically natural asymptotic limit of the system, have been explored before \cite{Wang:2009, jamison:2010}.  Our analysis differs from prior work in our SDE approach to describing the evolution of the configuration of the motors over time.

By inspection of the cargo's stationary distribution $\pi_Z$ from \eqref{eq:cargo-distr}, for a given motor configuration $\vec{x} = (x_1, x_2, \ldots, x_n)$, its mean position is $\znf(\vec{x}) = \frac{1}{\Nmot}\left[\big(\sum_1^\Nmot x_i\big) - \Ftrnd\right]$.  For our simplified model with non-fluctuating cargo, we simply assume the cargo instantaneously adjusts to this mean position as the motor coordinates evolve.  We denote the dynamics of the motors with non-fluctuating cargo by $\{\xnf(\tnf)\}_{i \in \{1,N\}}$, which satisfy the system of SDE
\begin{equation} \label{eq:defn-xnf}
	\difd \xnf_i(\tnf) = g\big(s[\xnf_i(\tnf) - \znf(\xnfvec(\tnf))]\big) \difd \tnf + \sqrt{\rho} \difd W_i(\tnf).
\end{equation}

For a two motor system, the cargo position is given by $\znf(\vec{x}) = (x_1 + x_2)/2 - \Ftrnd/2$. It follows that the nondimensionalized \emph{opposing} force felt by the first motor is $s(x_1 - \znf(\vec{x})) = s (x_1 - x_2 + \Ftrnd)/2$. The opposing force felt by the other motor is $s(x_2- \znf(\vec{x})) = s (x_2 - x_1 + \Ftrnd)/2$. To simplify notation, we introduce a new function $G(r) := g(-sr/2)$ which represents -- in the zero external force setting -- the instantaneous drift induced on a motor when the other motor is a given signed distance away.  That is, letting $r = x_1 - x_2$, (which is to say that $x_2$ is a signed distance $-r$ from $x_1$), the drift experienced by the first motor is $g(s(x_1 - x_2)/2) = G(-r)$.  When there is an external applied force $\Ftrnd$, the drift of the first motor is $g(s(x_1 - x_2 + \Ftrnd)/2) = G(-r - \Ftrnd)$ and the drift of the other motor is 
$ g(s(x_2-x_1 + \Ftrnd)/2) = G(r-\Ftrnd) $.

By changing variables to center-of-mass $\mnf(\tnf)=\frac12 (\xnf_1(\tnf)+\xnf_2(\tnf)) $ and difference $\rnf(\tnf)=\xnf_1(\tnf) - \xnf_2 (\tnf) $ coordinates, we obtain
\begin{align}
\nonumber	\difd \mnf(\tnf) &= \frac{1}{2} \left[G(\rnf(\tnf) - \Ftrnd) + G(-\rnf(\tnf) - \Ftrnd))\right] \difd \tnf + \sqrt{\frac{\rho}{2}}\,\difd W_m (\tnf),\\
\difd \rnf(\tnf) &= -\left[G(\rnf(\tnf) - \Ftrnd) - G(-\rnf(\tnf) - \Ftrnd)\right]\, \difd \tnf + \sqrt{2\rho}\, \difd W_r (\tnf). \label{eq:defn-mnf-rnf}
\end{align}
where we have
introduced  $ W_m \equiv \frac{1}{\sqrt{2}} (W_1 + W_2 ) $ and $ W_r  \equiv \frac{1}{\sqrt{2}} (W_1 - W_2) $, which can be treated in law as independent standard Brownian motions.

The process $\rnf$ is independent of $\mnf$ and the drift term can be viewed as the derivative of a potential function. 
Therefore, introducing the notation $G_\pm(r;\theta) = G(r - \theta) \pm G(-r - \theta)$, we can express the density of the stationary distribution of motor separation process $R$ as
\begin{equation} \label{eq:defn-pi}
\pi_R(r;\Ftrnd)=C_R \exp\left[-\frac{1}{\rho}\int_{-\infty}^r G_{-}(r^\prime;\theta)  \difd r^{\prime}\right], \ \ -\infty< r< \infty, 
\end{equation}
where $C_R$ is a normalizing constant.  The midpoint process $\mnf(t)$, which we use as a proxy for the position of the entire system, is then a functional of the separation process $\rnf(t)$ plus a Brownian motion,
\begin{equation*}
	\mnf(\tnf) = \int_0^{\tnf} \frac{1}{2} G_{+}(\rnf(t');\theta) \difd t' + \sqrt{\frac{\rho}{2}}\, W_m (\tnf). \label{eq:intm}
\end{equation*}
Applying the law of large numbers and the ergodic theorem for diffusion processes, we can calculate the asymptotic average velocity $\lim_{\tnf \to \infty} \mnf(\tnf) / \tnf$ by integrating the drift of $\mnf$ against the stationary distribution $\pi_R$,
\begin{equation} \label{eq:vel-two-nf}
	\vnf{2}{\Ftrnd} := \lim_{\tnf \to \infty} \frac{\mnf(\tnf)}{\tnf} = \int_{\rbb} \frac{1}{2} G_{+}(r;\theta) \pi_R(r;\Ftrnd) \difd r.
\end{equation}

The asymptotic diffusivity is given as follows
\begin{equation} \label{eq:diff-two-nf}
	\dnf{2}{\Ftrnd} := \frac{\rho}{4}+ \int_{-\infty}^\infty   \left( \int_{-\infty}^r \left( \frac{1}{2} G_{+}(r';\theta) - \vnf{2}{\Ftrnd} \right) \pi_{R}(r';\Ftrnd) \difd r' \right)^2  \frac{1}{ \rho \pi_{R}(r;\Ftrnd) }  \difd r.
\end{equation}

This effective diffusivity is obtained by way of the central limit theorem for stationary stochastic differential equations \cite{kutoyants2004statistical}. We center the process $ \mnf $ about its mean  and rescale by the square root of time, and use Eq.~(\ref{eq:intm}) to write: 
\begin{align}
	\label{eq:diffusivity}
	\frac{1}{\sqrt{\tnf}} \left(\M(\tnf) - \vnf{2}{\Ftrnd} \tnf \right) &=  \frac{1}{\sqrt{\tnf}} \int_0^{\tnf} \frac12 G_+(\R(t');\Ftrnd)) - \vnf{2}{\Ftrnd} dt'  + \sqrt{\frac{\rho}{2\tnf}} W_m(\tnf).
\end{align}

The first term on the right converges in distribution to a Gaussian random variable with mean zero and variance
\begin{equation*}
4 \int_{-\infty}^\infty   \left( \int_{-\infty}^r (\frac{1}{2}  \gtwoavg_+(r';\theta)- \vnf{2}{\Ftrnd} ) \pi_R(r';\Ftrnd) \difd r' \right)^2  \frac{1}{2 \rho \pi_R(r;\Ftrnd) }  \difd r
\end{equation*}
by the central limit theorem given in Chapter 1 of \cite{kutoyants2004statistical}.

Together with the Brownian motion term in \eqref{eq:diffusivity}, which is independent of $R(t)$,  we conclude that the external force dependent diffusivity is given by \eqref{eq:diff-two-nf}.  Thus for a given large $\tnf$, 
$$
\M(\tnf) \sim \Driftsys(\Ftrnd) \tnf+\sqrt{2 \dnf{2}{\Ftrnd} \tnf} \  Z,
$$
where $Z$ is a standard normal random variable.

\subsection{Comparison of One-Motor and Two-Motor Systems}
\label{sec:compare-no-fluc}

We are now ready to characterize how the spatial distribution of the motors affects the effective transport properties relative to the simpler models~\cite{Muller:2008pnas,Muller:2008jstatphys} in which all bound motors are assumed to share the load equally.

When there is no external forcing, we can understand the interference between the motors via the following thought experiment: Suppose the motors are separated by a distance $r$ and that the cargo is fixed at the mid-point between the motors. Each spring connecting the motors will be stretched equally inducing an opposing force with signed magnitude of $\pm sr / 2$ on the leading and trailing motors respectively (positive opposing force on the leading motor; negative opposing force on the trailing motor). The force-velocity curves in the literature~\citep{Visscher:1999,Kojima:1997,Shtridelman:2008,Shtridelman:2009} seem to indicate that a ``helpful'' force can only speed up the motor a small amount, whereas the same force applied in the opposite direction can slow the motor down considerably.  Therefore the mean of the two induced velocities is less than the velocity of a single motor, as characterized by the inequality
\begin{equation*}
	\frac{1}{2} (g(rs/2) + g(-rs/2)) < g(0), \qquad r \neq 0.
\end{equation*}
This is the motivation for emphasizing the concavity properties of $g$ in Assumption \ref{a:g-qual}. 

In the presence of an external force $ \Ftrnd$, the average velocity of the two motors will be less that that of one motor when $\Ftrnd $ and the separation distance $r$ satisfy 
\begin{equation} \label{eq:g-two-slower-nf}
	\frac{1}{2} (g(s(r + \Ftrnd)/2) + g(s(-r + \Ftrnd)/2) < g(s\Ftrnd).
\end{equation}
For the class of $g$ functions typical of force-velocity curves (such as our schematic example in Subsection~\ref{sec:g-and-h}), this inequality will only hold for relatively small values of external force $ \Ftrnd $ and a possibly restricted range of separation distances $ r $.  For example, note that when $r = 0$, the left-hand side is $g(s\Ftrnd/2)$ which is actually \emph{greater than} $g(s\Ftrnd)$ when $\Ftrnd  > 0$. 
Nevertheless, we must average the left hand side of \eqref{eq:g-two-slower-nf} against the stationary distribution of the motor separation $r$ \eqref{eq:defn-pi}, to attain a comparison of the average velocities of one-motor and two motor systems.  This explains the restrictions in conditions under which we can be confident that two motors bound to a cargo will be slower than one.

\begin{prop}[Two motors can be slower than one] 
	\label{thm:qual-nf}
	Suppose that the force-velocity function $g$ satisfies Assumption \ref{a:g-qual}. Then for any given value of the ratio of typical spring force to stall force, $s$, there exists a positive value $ \Ftrnd_c (s) $ such that for all  $ \Ftrnd < \Ftrnd_c (s)$,
	\begin{equation} \label{eq:low-force-nf}
		\vnf{2}{\Ftrnd} < \vnf{1}{\Ftrnd}. 
	\end{equation}
\end{prop}
\begin{proof}
	In light of \eqref{eq:velocity-one-motor}, the nondimensionalized velocity of a system with one motor is given by $\vnf{1}{\Ftrnd} = g(s \Ftrnd)$. By  \eqref{eq:vel-two-nf} and the definition of $G$, the inequality \eqref{eq:low-force-nf} holds whenever
	\begin{equation} \label{eq:integrand-inequality}
		\int_\rbb \frac{1}{2} \big(g(s(r + \Ftrnd)/2) + g(s(-r + \Ftrnd)/2)\big) \difd \pi_{R}(r;\Ftrnd) < g(s\Ftrnd).
	\end{equation}
As discussed earlier, the inequality \eqref{eq:g-two-slower-nf} does not hold for all $r$, so we must establish the inequality for $\Ftrnd = 0$ and then rely on continuity of the functional with respect to $\Ftrnd$. 
	Indeed, we employ Assumption \ref{a:g-qual} \eqref{a:g-strong} to find that
	\begin{align*}
		\int_\rbb \frac{1}{2}(g(sr/2) + g(-sr/2)) \difd \pi_R(r;0) < \int_{0}^\infty g(0) \difd \pi_R(r;0) = g(0).
	\end{align*}
 	To extend the inequality to an open neighborhood of the origin, we note that $g(s\Ftrnd) - \vnf{2}{\Ftrnd}$ is continuous in $\Ftrnd$ and so must remain positive on some open interval containing the origin.		
\end{proof}

At large applied forces, the appropriate comparison is between the stall force of two motors versus twice the stall force of one motor. First it is worth recalling that in nondimensional terms, the stall force is $\Ftrnd = 1/s$ (see discussion in Section \ref{sec:nondim}).  Therefore we compute the velocity for $\Ftrnd = 2/s$ and demonstrate that it is positive. Indeed at a separation $r$, the velocity of the system satisfies 
\begin{equation*}
	\frac{1}{2}(g(s(r + \Ftrnd)/2) + g(s(-r + \Ftrnd)/2)) = \frac{1}{2} (g(1+sr/2) + g(1-sr/2)) > g(1)
\end{equation*}
where $g(1) = 0$ by hypothesis.

The intuition here is that while the leading motor is actually moving backward because it is experiencing a force greater than its stall force, the trailing motor is still moving forward, and doing so with a magnitude greater than that of the leading motor.  The mean of the two velocities is therefore positive and the system has an overall positive drift. In other words, it takes more than twice the stall force of a single motor to stall the two-motor system. 

\begin{prop}[Superadditive stall forces] 
	\label{thm:superadd-nf}
	Suppose that the force-velocity function $g$ satisfies Assumption \ref{a:g-qual} and that the stallibility satisfies $s < 1$. Let $\Ftrnd_1^*$ and $\Ftrnd_2^*$ denote the minimum nondimensionalized force necessary to result in an average velocity of 0 for a cargo attached to one or two motors, respectively:
	\begin{equation*}
	\Ftrnd_j^* \equiv \min\{\Ftrnd:  \vnf{j}{\Ftrnd} \leq 0\}.
	\end{equation*}
	Then
	\begin{equation} \label{eq:superadditive-stall}
		\Ftrnd_2^* > 2\Ftrnd_1^*
	\end{equation}
\end{prop}
\begin{proof}
	For a given $\Ftrnd$, we recall $\vnf{1}{\Ftrnd} = g(s\Ftrnd)$. By Assumption 1, we can directly observe that $\Ftrnd_1^* = 1/s$. To show \eqref{eq:superadditive-stall} recall the definition of $\theta_1$ from Assumption \ref{a:g-qual} \eqref{a:g-strong}. It follows that for all $\Ftrnd > 2 \theta_1/s$, 
	\begin{align*}
		\vnf{2}{\Ftrnd} &= \int_\rbb \frac{1}{2} \big(g(s(r + \Ftrnd)/2) + g(s(-r + \Ftrnd)/2)\big) \difd \pi_{R}(r;\Ftrnd) \\
		&> \int_\rbb g(s\Ftrnd/2) \difd \pi_{R}(r;\Ftrnd) = g(s\Ftrnd/2) = \vnf{1}{\Ftrnd/2}
	\end{align*} 
	In particular, since $s < 1$ we have $\theta_1 < 1 < 1/s = \Ftrnd_1^*$. Therefore for all $\Ftrnd \in (2\theta_1/s, \Ftrnd_1^*]$, we have $\vnf{2}{\Ftrnd} > \vnf{1}{\Ftrnd} \geq 0$ implying \eqref{eq:superadditive-stall}.
\end{proof}

Because the the strict superadditive inequality relies on averaging against the stationary distribution $\pi_R$ of the motor configurations, we see that stochasticity in motor position is crucial to this result.  Without these fluctuations, the motors would simply stall at $\Ftrnd = 2/s$ because each motors' share of the total force would be exactly their stall force.  The random component of the stepping makes it possible for the motors to get out of a stalled configuration.  Once separated, the differentiated responses of the leading and trailing motors makes progress possible for the motor-cargo system.

The condition $ s < 1 $ means that the typical spring force is less than the stall force of the motor, so that the fluctuations in the cargo by themselves do not typically subject the motors to stalling forces.  
We remark that similar conclusions relating the properties of the two motor and cargo system to the convexity properties of the single-motor force-velocity curve, also for the case of non-fluctuating cargo, were developed by \cite{Wang:2009} in the context of an exact solution of a model in which the motors were represented as random walkers on a lattice, coupled by their attachment to the cargo.  Their model did not include cargo fluctuations, but in the next section we will extend our analysis to include these.

Before moving on, we note a contrast to models that do incorporate motor unbinding.  As discussed earlier, part of our objective is to determine which properties of multi-motor transport can be directly attributed to distribution of forces when the spatial configuration of the motors is accounted for. It has been shown, however, that in certain regimes motor detachment is the dominant feature.  

For example, Kunwar et al \cite{kunwar2008stepping,Kunwar:2010} observed that when the cargo is subjected to a significant force, the majority of this force is borne by the leading motor, which is therefore more likely to detach than the trailing motor.  When detachment occurs, the system snaps back to the location of the trailing motor.  Often this other motor, now bearing all the force of the optical trap, will detach as well. (This detachment cascade is crucial to the tug-of-war dynamics described by M\"uller, Klumpp and Lipowsky \cite{Muller:2008pnas}.)  If the trailing motor does not detach and it is assumed that the leading motor will reattach not too far from the position of the trailing head, then the system will proceed after having lost some ground. This sequence of events is common enough that Kunwar et al report that the stall force of a two motor system is \emph{subadditive} when compared to the stall force of a single motor system. By contrast, our non-detachment model predicts a \emph{superadditive} stall force. This is an important distinction to make because tug-of-war systems may spend substantial periods of time in a ``stalemate'' between forward and backward motors that are stalled, but are not exerting sufficient force to induce detachment \cite{Muller:2008pnas}.  

\subsection{Analysis including cargo dynamics}
\label{sec:analysis-two-motors-fluct}

Having developed some of the key two-motor results for the simplified setting neglecting cargo fluctuations, we now proceed to extend them to include the cargo fluctuations.  To write a version of Theorem \ref{thm:qual-nf} that includes thermal fluctuations of the cargo, we redefine the function $G(r)$, which represents the instantaneous drift of a motor when the other is a signed distance $r$ away. Rather than being written for a fixed cargo position, $G$ now includes an average over the cargo position.
We redefine the function $\gtwoavg(r)$
as follows. From \eqref{eq:g-avg} and \eqref{eq:cargo-distr},
\begin{align}
	\gtwoavg(r) &\equiv \gavg_1(\vec{x}; 0) = \gavg_2 (\vec{x};0) = \gaussbigbig{\frac{-rs}{2}}{\frac{s^2}{4}}{g}. \label{eq:gwithflucs}	
\end{align}

\begin{figure}[ht]
\begin{center}
\includegraphics[scale=.35]{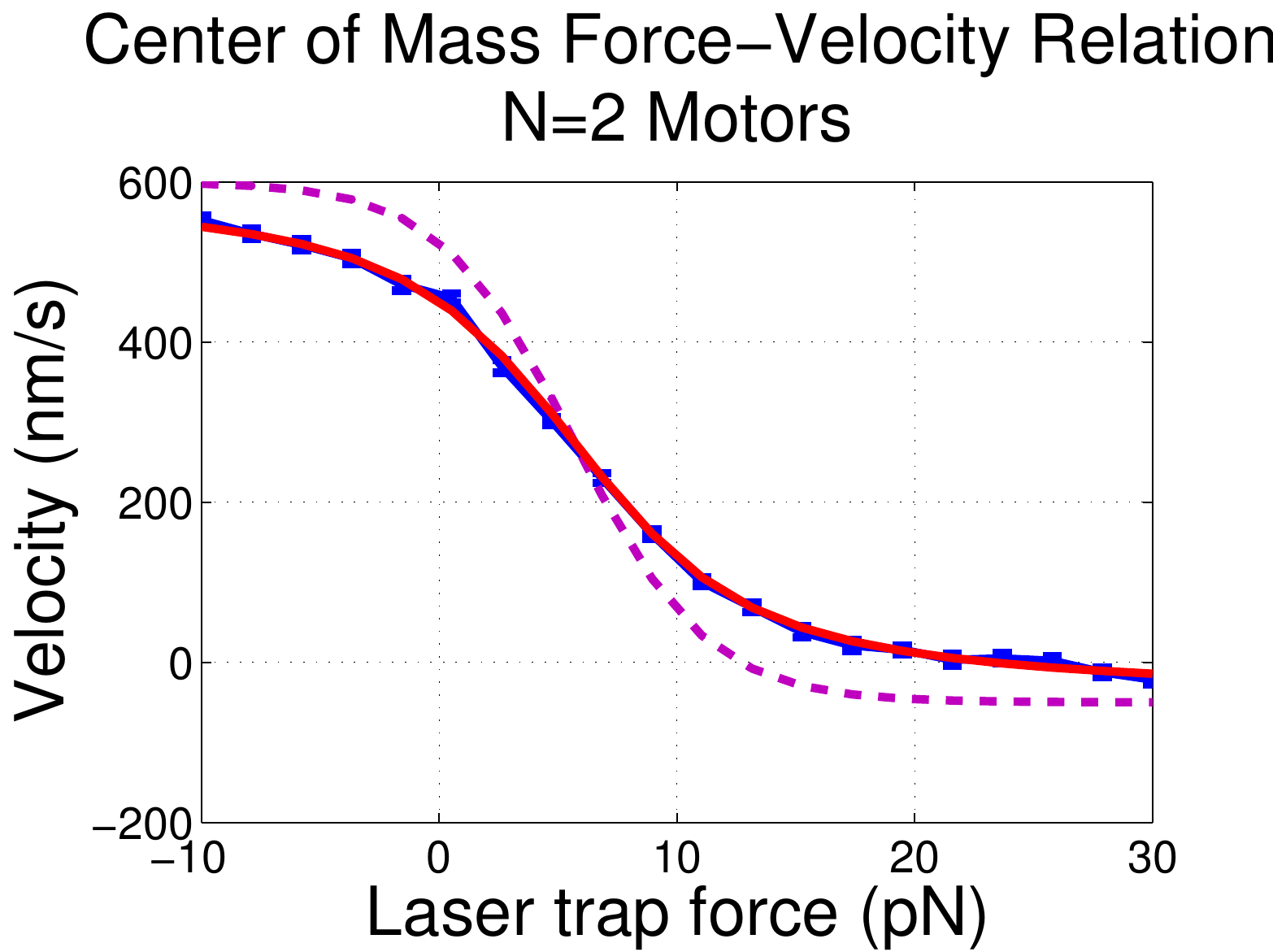}
\includegraphics[scale=.37]{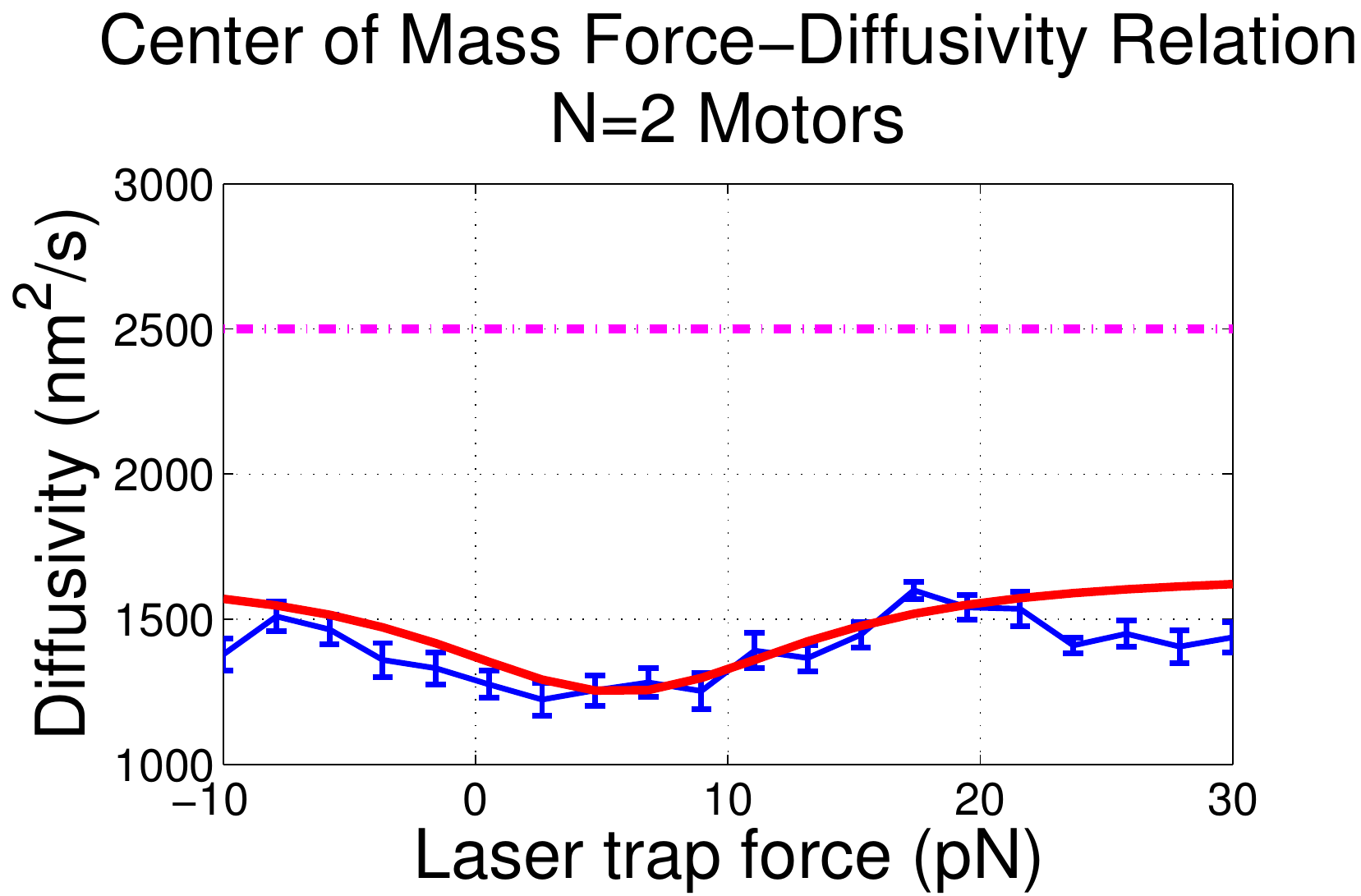}
\caption{\label{fig:cmforcevelocity} Comparison of our stochastic averaging results with simulation of the effective force-velocity (left) and force-diffusivity (right) relationships of two cooperative motors attached to a common cargo.  The direct stochastic simulations of the model equations \eqref{eq:motorcargosde} appear in blue with error bars denoting sample standard deviation. The stochastic averaging results, appearing as red solid curves, are computed from Eqs.~\eqref{eq:vel-two-nf} and \eqref{eq:diff-two-nf} for the velocity and diffusivity respectively, with $G$ defined by \eqref{eq:gwithflucs}. The force-velocity curves are compared to the prediction of the force-balance theory ($v g(\theta/2F_*)$, magenta and dashed).  The force diffusivity curves are compared to the diffusivity of the cargo when attached to  one motor ($\sigma^2/2$, magenta and dashed).
 Parameter values are listed in Table~\ref{tab:const}, and we take the model (\ref{E:gbar-form}) for the instantaneous force-velocity relationship for a single motor.}
\end{center}
\end{figure}

The calculation for  the asymptotic average velocity $\vf{2}{\Ftrnd}$ and diffusivity $\df{2}{\Ftrnd}$ proceeds exactly as in Section \ref{sec:nocargofluc}, with the effects of the cargo fluctuations appearing only in the redefinition of $ \gtwoavg $.
The accuracy of our stochastic averaging formula (Eqs.~(\ref{eq:vel-two-nf}) and (\ref{eq:gwithflucs})) for the average velocity of two cooperative motors bound to a cargo is demonstrated in Figure~\ref{fig:cmforcevelocity} by its excellent agreement with direct numerical simulations of the underlying explicit stochastic dynamical model for the motor and cargo positions (Eq.~\ref{eq:motorcargosde}).  We also contrast our theoretical force-velocity curve with what would be obtained by a simpler force-balance theory, in which the motors are always assumed to bear half the force applied to the cargo, 
taking into account neither cargo fluctuations nor the varying spatial configuration of the two motors~\citep{Muller:2008pnas,Muller:2008jstatphys}.  We see that incorporating stochastic fluctuations of the cargo and relative motor positions does create a substantial change in the theoretical prediction of the effective transport of the motor-cargo system, even in the simplified context of two identical cooperative motors without binding and unbinding dynamics (also noted for a discrete stepping model without cargo fluctuations in~\citet{Wang:2009})
 The effective diffusivity, seen in the panel on the right in Figure~\ref{fig:cmforcevelocity}, for the cargo attached to two motors is slightly more than half  the value $ \frac{\sigma^2}{2} $ of the effective diffusivity when attached to a single motor (\ref{eq:one-motor-diffusivity}).  The effective diffusivity increases somewhat as the applied force increases in either direction.

\begin{figure}[ht]
\begin{center}
\includegraphics[scale=.35]{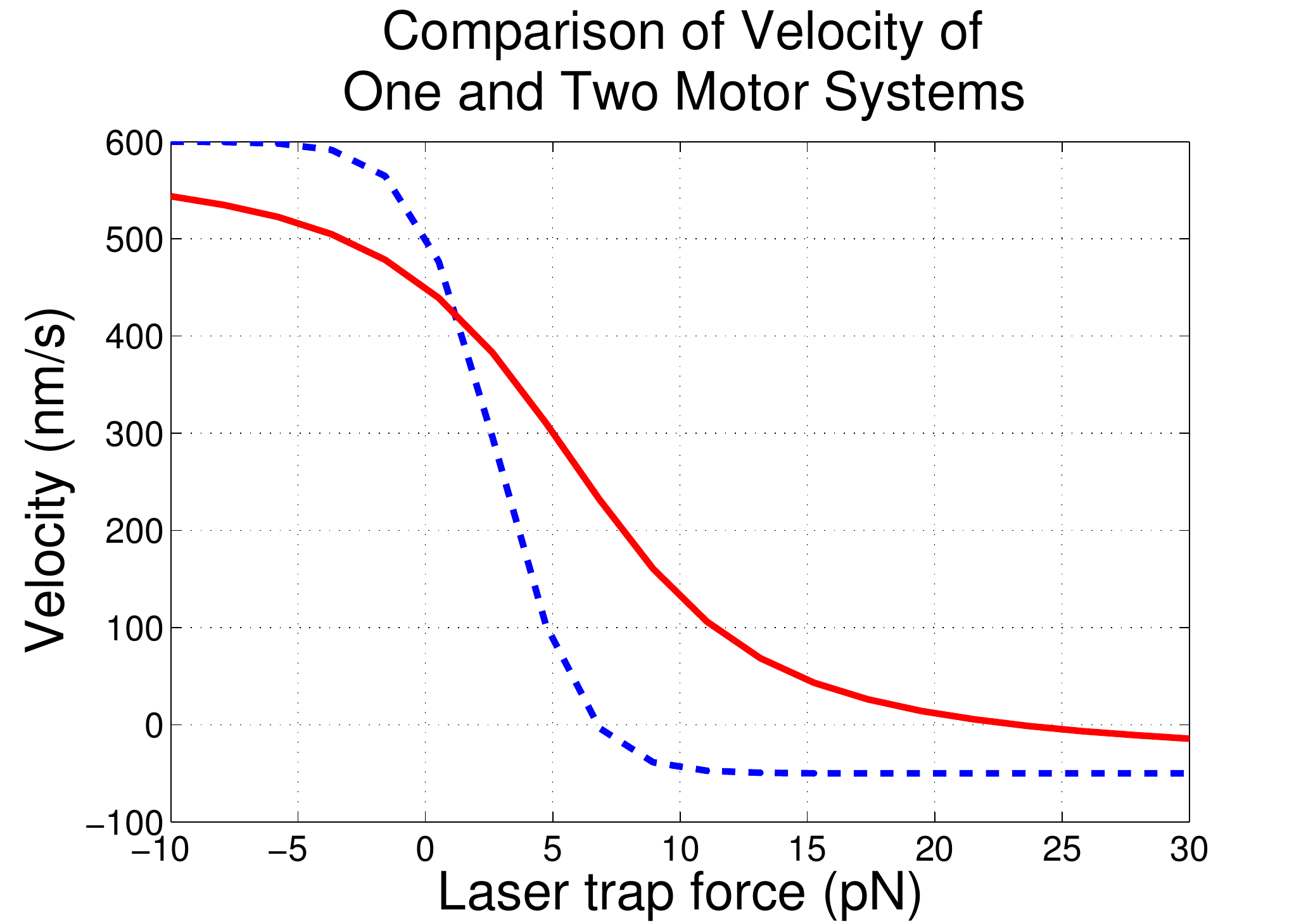}
\caption{\label{fig:onevstwo} Effective force-velocity curves for a single motor attached to a cargo (blue dashed, from Eq.~(\ref{eq:velocity-one-motor})) and two cooperative motors attached to a common cargo (red solid, from 
(Eqs.~(\ref{eq:vel-two-nf} and (\ref{eq:gwithflucs}) ).  Parameter values are listed in Table~\ref{tab:const}, and we take the model (\ref{E:gbar-form}) for the  instantaneous force-velocity relationship for a single motor.}
\end{center}
\end{figure}

Figure~\ref{fig:onevstwo} compares the effective force-velocity relationship of the cargo when bound to one or two motors.  
In particular, we observe that the system with two motors is moving slower at low load forces than the single-motor-cargo system.   On the other hand, at large applied forces, the two-motor-cargo system moves faster, and we see that for our parameter choices, the stall force for the two-motor-cargo system is about three times that of the single-motor-cargo system.  
These qualitative features have been previously noted in various models~\cite{kunwar2008stepping,Wang:2009,jamison:2010}, usually through direct numerical simulations.  Our analytical formulation allows us to provide more mathematical support to these observations.

We begin by showing that $G$ satisfies a concavity result similar to $g$.
\begin{lemma} \label{lem:G-concavity}
	Let $g$ satisfy Assumption \ref{a:g-qual}. Then there exists an $s_* > 0$ such that for any $s \in (0,s_*)$ we have that for all $r > 0$
	\begin{equation} \label{eq:G-concave}
		\frac{1}{2}\big(\gtwoavg(r) + \gtwoavg(-r)\big) < G(0).
	\end{equation}
\end{lemma}
\begin{proof}
	Note that
	\begin{align}
	\nonumber	& \gtwoavg(0) - \frac{1}{2}\big(\gtwoavg(r) + \gtwoavg(-r)\big) = \gaussbigbig{0}{\frac{s^2}{4}}{g} - \frac{1}{2}\big(\gaussbigbig{\frac{-rs}{2}}{\frac{s^2}{4}}{g} + \gaussbigbig{\frac{rs}{2}}{\frac{s^2}{4}}{g}\big) \\
	\nonumber & \qquad \qquad = \sqrt{\frac{2}{\pi s^2}} \int_{\rbb} \left[g(y) - \frac{1}{2} \left(g\Big(y + \frac{rs}{2}\Big) + g\Big(y - \frac{rs}{2}\Big)\right)\right] \exp\Big(\frac{-2 y^2}{s^2} \Big) dy \\
	& \qquad \qquad = \sqrt{\frac{2}{\pi s^2}} \int_{0}^\infty \left[\eta(y) - \frac{1}{2} \left(\eta\Big(y + \frac{rs}{2}\Big) + \eta\Big(y-\frac{rs}{2}\Big)\right)\right] \exp\Big(\frac{-2 y^2}{s^2} \Big) dy \label{eq:eta-integral}
	\end{align} 
where we recall the function $\eta(y) = \left(g(y) + g(-y)\right)$ from Assumption 1, which by hypothesis is decreasing for positive $y$ (and therefore by even symmetry as a function of $ |y| $).  

For each fixed and suitably small $s$, we produce a lower bound for the right-hand side in two cases:  1) $rs < f_*$ and 2) $rs > f_*$, where $f_*$ is the constant from Assumption~\ref{a:g-qual} such that $g''(y) < 0$ for all $y < f_*$. 

To address the first case, we rely on the concavity properties of  $g$ to deduce there exists an $m > 0$ such that for $y \in [-f_*,f_*]$, $\eta''(y) < -m$.  Since $y \in [0,f_*/2]$ implies for this first case that $ -f_* <  y-rs/2 < y+rs/2  <f_*$, we obtain after a second order Taylor expansion:
\begin{equation*}
	\eta(y) - \frac{1}{2}\left[\eta\left(y-\frac{rs}{2}\right)+\eta\left(y+\frac{rs}{2}\right)\right] > \frac{r^2s^2}{8} m \text{ for } y \in [0,f_*/2].
\end{equation*}
For all other $y$, we use the upper bound 	$|\eta(y) - \frac{1}{2}(\eta(y-rs/2)+\eta(y+rs/2))| < \frac{r^2s^2}{8} |g''|_\infty$, where $|\cdot|_\infty$ is the sup norm over $\mathbb{R}$.

Splitting the integral in \eqref{eq:eta-integral} into two parts $I = I_1 + I_2$ with the integrals $I_1$ and $I_2$ being taken over $[0,f_*/2)$ and $[f_*/2,\infty)$ we have the estimate
\begin{equation*}
	\text{RHS \eqref{eq:eta-integral}} > \frac{r^2s^2}{8} \left[\left(\frac{1}{2} - \P\left\{Z_s > \frac{f_*}{2}\right\}\right)m -  \P\left\{ Z_s  > \frac{f_*}{2}\right\} |g''|_\infty \right]
\end{equation*}
where $Z_s$ is a normal random variable with mean 0 and variance $s^2/4$.  We can therefore choose an $s_1 > 0$ (independent of $ r $) such that for all $s < s_1$, the right hand-side is positive.

The Taylor approximation is not useful when $r$ is not small, and so in the second case $rs > f_*$, we seek a uniform bound. Because $\eta$ is a decreasing function of $ |y| $, we see that for each fixed $y $,  $\eta(y-a) + \eta(y+a)$ is a decreasing function of $ a $ for $ a > |y| $.  Consequently, whenever $ 0 \leq y \leq f_*/2 $ and $ rs > f_* $, we can deduce that $\eta(y) - (\eta(y-rs/2) + \eta(y+rs/2))/2 \geq \eta(y) - (\eta(y-f_*) + \eta(y+f_*))/2 \geq K$ for some constant $ K > 0 $ independent of $ r $.
	
Splitting the integral	in \eqref{eq:eta-integral} as in the first case, and combining the estimates from the previous paragraph with the global bound $\eta(y) - \frac{1}{2}(\eta(y-rs/2) + \eta(y+rs/2)) < 4 |g|_\infty$, we obtain for the second case ($rs > f_*$):
	\begin{equation*}
		\text{RHS \eqref{eq:eta-integral}} > \left(\frac{1}{2} - P\Big\{Z_s > \frac{f_*}{2}\Big\}\right) K - 4 \P\Big\{Z_s > \frac{f_*}{2}\Big\} |g|_\infty 
	\end{equation*}
	Once again, there exists an $s_2 > 0$ (independent of $ r $) such that for all $s < s_2$ the right-hand side is positive.  Taking $s_* = \min(s_1,s_2)$ completes the proof.
\end{proof}

With this lemma we can now prove the following.
\begin{theorem} \label{thm:qual}
	Let $g$ satisfy Assumption \ref{a:g-qual}. Then there exist $s_*,\theta_* > 0 $ such that for any $s < s_*$ and $ |\Ftrnd| < \theta_* $, we have  $\vf{2}{\Ftrnd} < \vf{1}{\Ftrnd}$.
\end{theorem}
\begin{proof}
We proceed as in the proof of Theorem \ref{thm:qual-nf}: proving the property for $\Ftrnd = 0$ and then extending it by continuity in $\Ftrnd$ of an appropriate functional.  Taking the computation of the average velocity from Section \ref{sec:nocargofluc} and applying \eqref{eq:G-concave} yields
\begin{equation*}
	\vf{2}{\Ftrnd} = \frac{1}{2}\int_\rbb (\gtwoavg(r) + \gtwoavg(-r)) \pi_R(r;0) dr < \int_\rbb G(0) \pi_R(r;0) dr = G(0)
\end{equation*}
Now, $\gtwoavg(0) = \gaussbig{0}{\frac{s^2}{4}}{g}$, which we can rewrite as
\begin{align*}
	\gaussbigbig{0}{\frac{s^2}{4}}{g} &= \frac{1}{\sqrt{\pi}} \int_0^\infty \left(g(sy/\sqrt{2}) + g(-sy/\sqrt{2})\right) e^{-y^2} dy
\end{align*}
Since $g$ satisfies Assumption \ref{a:g-qual} \eqref{a:g-strong}, the integrand is decreasing and so
\begin{align*}
	\gtwoavg(0) &< \frac{1}{\sqrt{\pi}} \int_0^\infty \left(g(sy) + g(-sy)\right) e^{-y^2} dy
	= \gavg(0) = \vf{1}{0}.
\end{align*}
\end{proof}

\subsection{Comparison to the high viscosity regime}
\label{sec:twomotor-highvisc}

The results of the preceding sections establish that there is a clear interference effect when multiple motors are involved in cargo transport.  However, the superadditivity of the stall force implies that when significant force is applied to the cargo, then multiple motors can cooperate and improve transport performance. It remains to discuss the performance of two motors in the presence of a high viscosity environment, in particular to address whether the forces encountered \emph{in vivo} are sufficient so that two motors perform better than one.  To this end, we show in Figure~\ref{fig:oneandtwodrag} the results of direct numerical simulations of the model (\ref{eq:motorcargosde}) for higher values of the cargo friction parameter $ \gamma $ which might correspond to \emph{in vivo} conditions.  We see that the stall force for the two-motor-cargo system remains about the same in our model ($ \gtrsim 20 pN $), about three times the stall force for a single motor.  At low applied forces, the two-motor-cargo speed is even slower than that predicted by our asymptotic (small $ \epsilon$) theory, and so is in particular slower than a single motor.  Thus, the qualitative conclusions of our analytical theories for small cargo friction seem to remain true for the larger frictions expected \emph{in vivo}.

We do not yet have a direct analytical approach to identify the threshold viscosity that constitutes a boundary between the diffusion dominated and drag dominated regimes.   But by comparing the direct numerical solutions to the 
 the prediction of the stochastic averaging theory in Figure \ref{fig:oneandtwodrag}, we see that the unloaded cargo velocity begins to depart from the asymptotic theory when $ \epsilon \sim 10^{-1} $ (top panel), whereas the overall structure of the force-velocity curve for a cargo attached to two motors shows significant deviations from the asymptotic theory only as $ \epsilon \sim 1 $ (lower panels).  In either case, our stochastic averaging theory remains superior to the force balance theory.

This is significant because when there is no external force applied and the drag force is $\gamma = 10^{-2}$ pN s$/$nm, we observe from Figures \ref{fig:oneandtwodrag} that the velocities on one motor and two motor systems is roughly 275 and 350 nm$/$s, respectively. Despite concerns that there may be interference between motors \emph{in vivo}, it is more likely that the drag forces are sufficient that two motor transport is indeed faster.

\begin{figure}[ht]
\begin{center}	
\includegraphics[scale = 0.4]{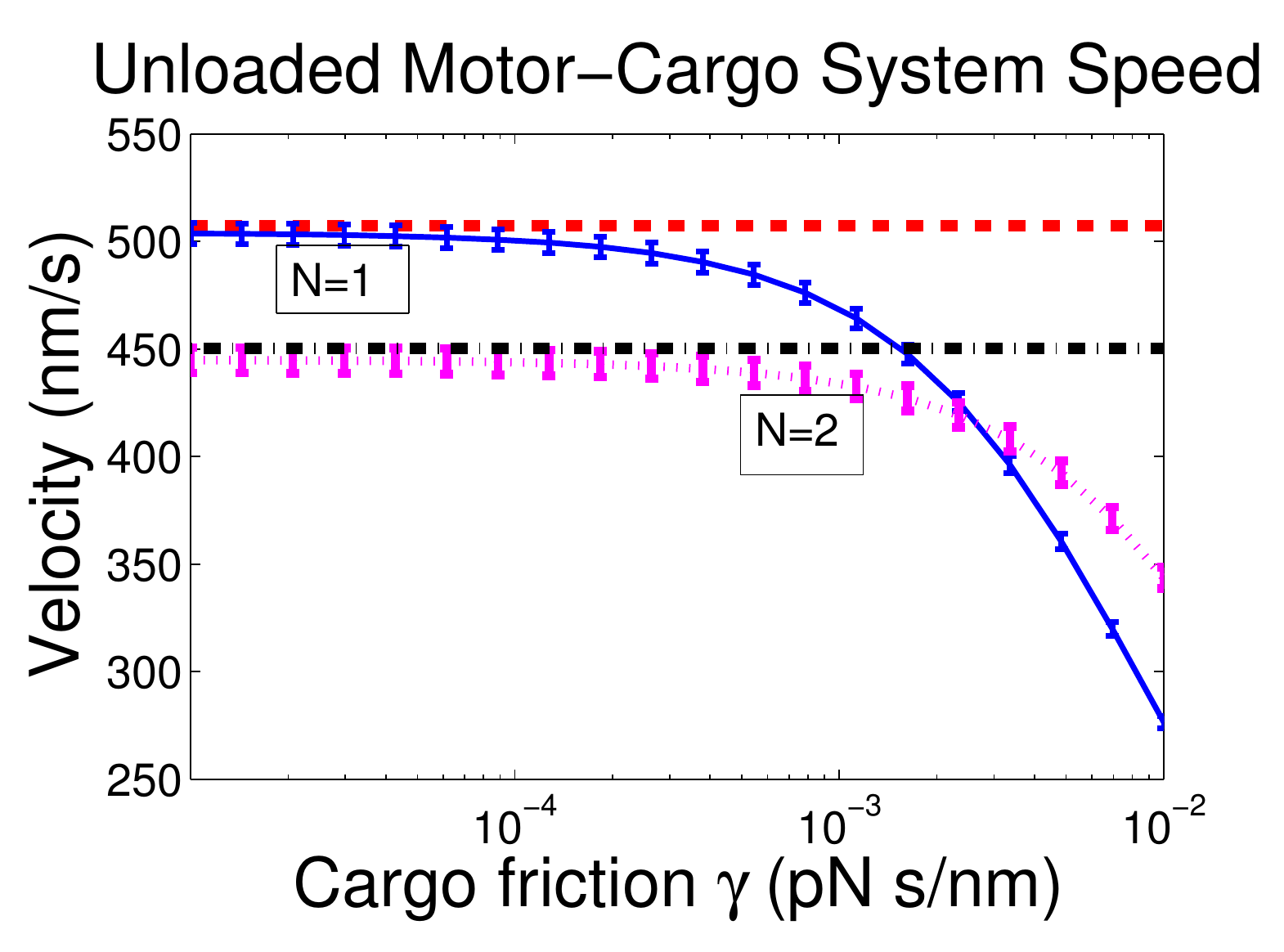} \\
\includegraphics[scale=.35]{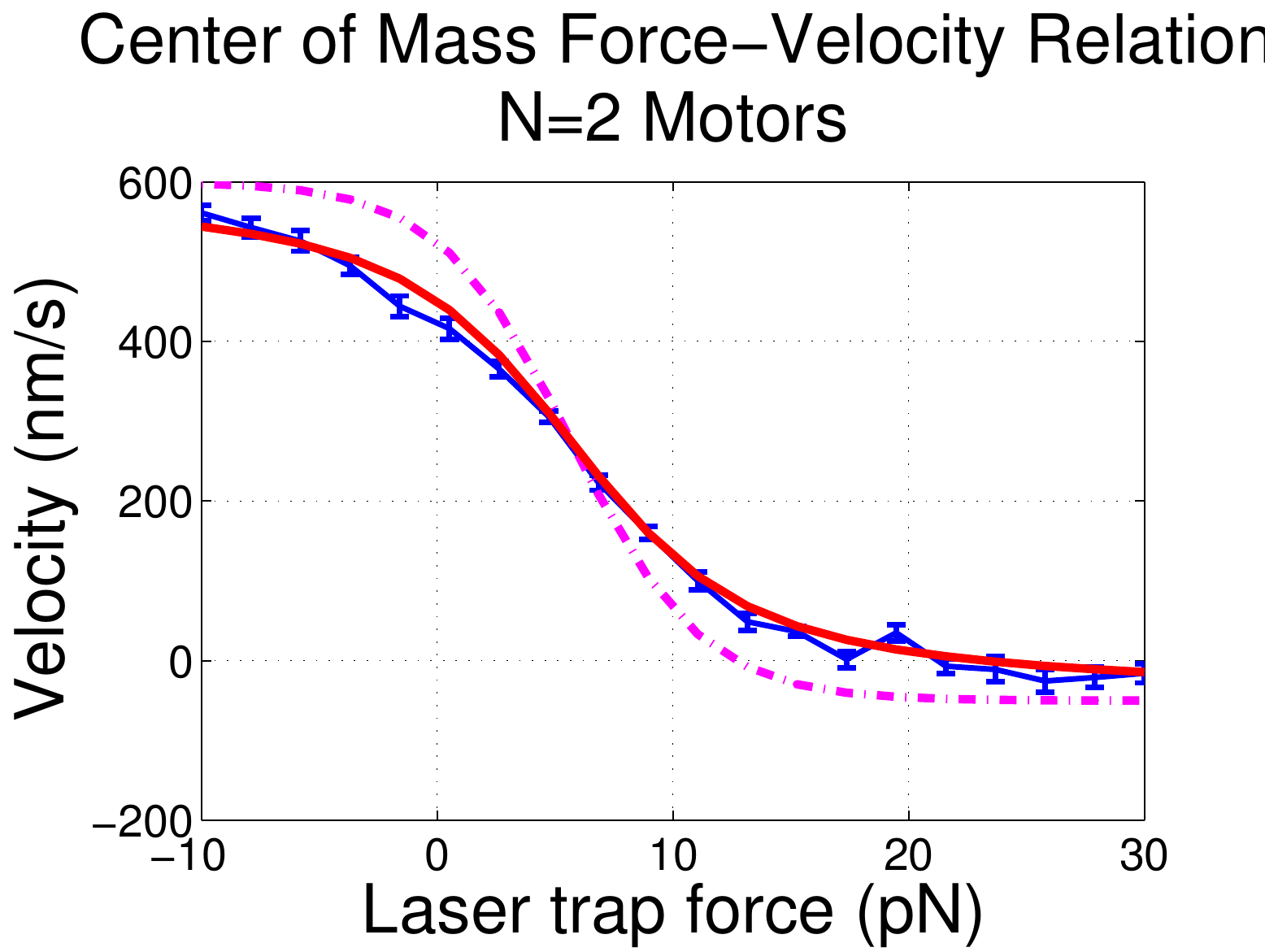} \hspace{0.2 in}
\includegraphics[scale=.35]{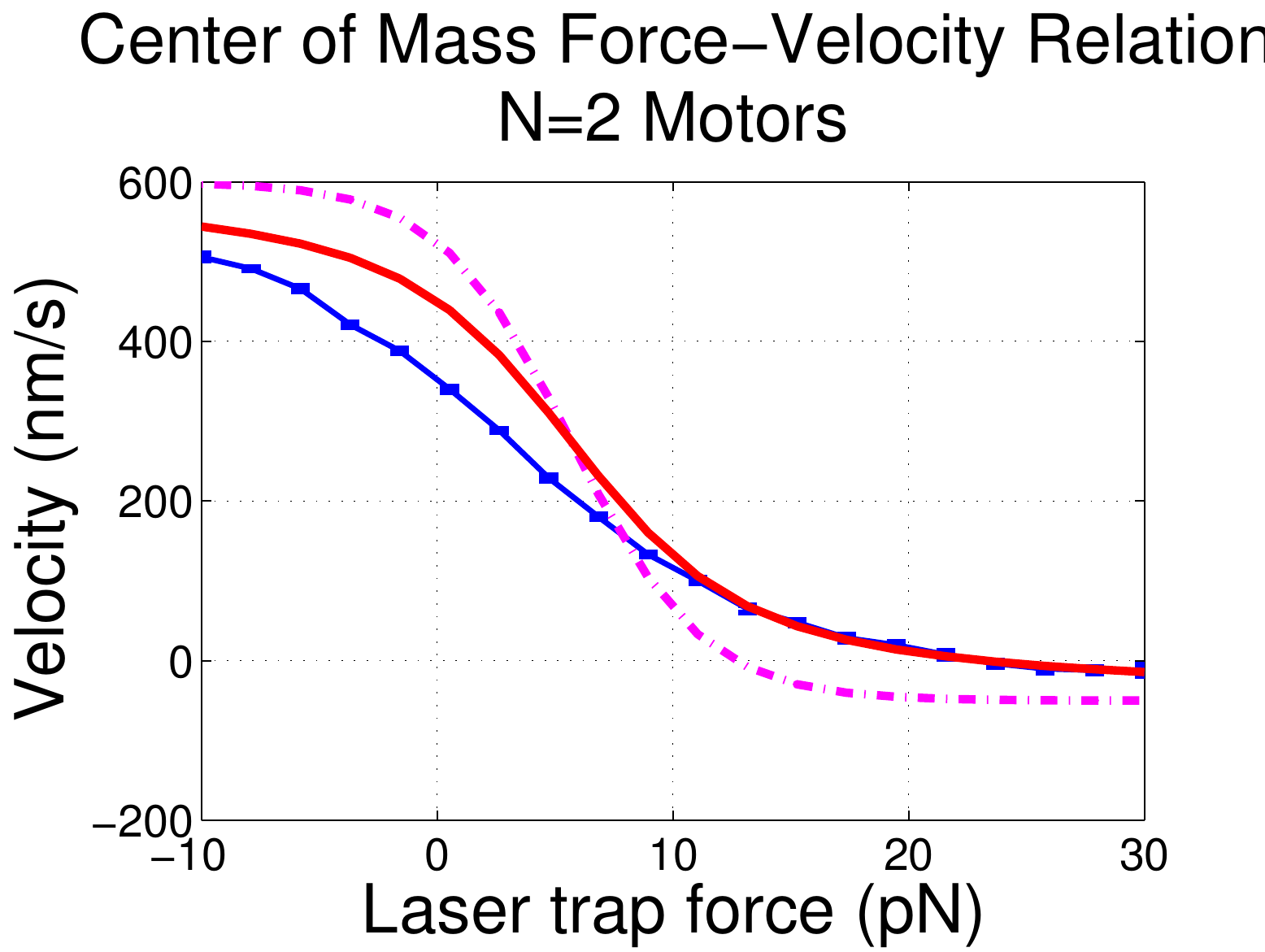}
\caption{\label{fig:oneandtwodrag} Top panel: Velocity of cargo attached to one or two motors as function of cargo friction parameter $ \gamma $.  Results of direct simulation of (\ref{eq:motorcargosde}) are indicated by the blue solid curve (one motor) and magenta dashed curve (two motors), with error bars representing one sample standard deviation.  The red dashed line (one motor) and black dot-dashed line (two motors) indicate the theoretical values corresponding to the $ \epsilon \downarrow 0 $ limit, when the time scale of fluctuation of the cargo is much smaller than the time scale of the motor motion.  The values of $ \gamma $ presented correspond to a bead of size $ a = 500 nm $~\citep{guydosh2006backsteps} and solvent viscosity ranging from the viscosity of water (at 300 K) to 1000 times that value
(so $ \epsilon $ ranges from $ 3 \times 10^{-3} $ to $ 3 $).  Lower panels: Numerical assessment of validity of stochastic averaging theory at higher solvent viscosities (left $ \gamma = 10^{-3} \text{pN} \cdot \text{s} / \text{nm} \qquad (\eta = 10^2 \eta_0) $; right $ \gamma = 10^{-2} \text{pN} \cdot \text{s} / \text{nm}\qquad  (\eta = 10^3 \eta_0) $
where $ \eta_0 = 10^{-9}  \text{pN} \cdot \text{s}/\text{nm}^2$ is the viscosity of pure water at room temperature).  The meaning of the curves is the same as in Figure~\ref{fig:cmforcevelocity}.
}
\end{center}
\end{figure}

\section{Conclusions}

We have developed a system of stochastic differential equations to describe the dynamics of multiple identical motors transporting intracellular cargo along a microtubule.  We use a Langevin construction for the cargo equation and a diffusion approximation for the motor equations.  

We perform a dimensional analysis of our system of equations, and we find important dimensionless groups of parameters; this reveals that the dynamics \emph{in vitro} are qualitatively distinct from what is likely to occur \emph{in vivo}.  In particular, the model suggests that in the experimental setting, the cargo is fluctuating rapidly relative to the slower moving motors.  

In addition to the stochasticity found in the cargo dynamics, we study the impact of fluctuating motor positions.  Using the randomness parameter calculated by \cite{Fisher:2001}, we note that the motors' positions relative to each other are constantly changing.  When multiple motors are engaged in transport of a single cargo, the force distribution among the motors is therefore state dependent. When the force-velocity relationship (Fig \ref{fig:gex}) is nonlinear, it follows that the transport properties of teams of identical motors are non-trivial.  We find that in the low viscosity environment, multiple motors are slower than one; however, for a fluid with viscosity similar to estimates of the cytoplasm \cite{Mitchell:2009}, drag force is sufficient that multiple motors perform better than a single motor. 

We also use our model to study the change in stall force as a function of motor copy number and viscosity of the fluid environment.  This stall force is calculated to increase superadditively for multiple motors and also increases with the viscosity of the fluid environment.  Neither of these results is in agreement with past work, and therefore a more careful study of motor dynamics near stall force is an important next step.  The inclusion of binding and unbinding events will also clearly lead to a substantial change in the calculated stall force. In addition, we expect that to account for richer dynamics, the diffusion term in the motor SDE may need to be amended to a multiplicative, rather than additive, form.

\section{Acknowledgments}
	The collaboration for this work was initiated during the Stochastic  Dynamics program at the Statistical and Applied Mathematical Sciences Institute (SAMSI), when the authors were partially supported as either long-term visitors or postdoctoral fellows.  The authors would further like to thank Will Hancock, Jonathan Mattingly, Michael Reed, John McSweeney and all the members of the SAMSI Stochastic Biological Dynamics Working Group for helpful conversations as we developed this work.  
Peter Kramer was partially supported by NSF CAREER grant DMS-0449717.  John Fricks was supported by the Joint DMS/NIGMS Initiative for Mathematical Biology DMS-0714939.

\appendix

\section{Extension to Nonlinear Spring Laws}
\label{sec:nonlinspring}

In the main text, we take a simple linear spring model (\ref{eq:lintail}) for the  force law for the tails connecting the motors to the cargo.  One could criticize this selection for a number of reasons.  The tail actually has a complex hinged structure 
and would likely go ``slack" under compression
(as in the model of~\citet{Korn:2009}).  Also, we recall that in Eq.~(\ref{eq:lintail}), $ \slen $  represents the separation between the motor and cargo projected along the microtubule, which is not the actual length of the tail in its real three-dimensional context.  Finally, since the cargo is much larger (typically $ \sim 500 $ nm) than the motor or tether, the geometry of the motor's tail and location of its attachment to the cargo may be significant~\citep{Korn:2009}.   The effective one-dimensional force law may depend on how far away from the microtubule the connecting tail is bound to the cargo, and since this is potentially different for different motors, each motor's tail may have a different force law!  
The purpose of this appendix is to show that our modeling framework and stochastic averaging procedure can robustly incorporate these more realistic complications.  

The reason we have not pursued them in the main text, beyond the desire for a simple presentation of the key ideas,  is that we could not identify a specific nonlinear force law that would be clearly better than the linear spring force law.
The accuracy of experimental observations is too coarse to motivate a more detailed model for the connection; in fact controversy persists over what qualitative form the simple spring model should take~\cite{Hariharan:2009}.  Some authors prefer a wormlike chain model~\cite{Howard:2001}
$ \Fspring (\slen) = \kappa \slen + \frac{\mathrm{sgn} (\slen)}{4} \left((1-|\slen|/\lcon)^{-2} -1\right) \text{ for } |\slen| < \lcon $
that stiffens as the effective spring length approaches a critical contour length $\lcon $, beyond which the tether becomes taut and will not stretch any further. The framework we propose allows a general nonlinear spring structure to cover both of these, as well as other models for the force response of the tether attaching the motor and cargo.

To incorporate a more general tail force law, which could differ for each motor, the model equations can be modified (\ref{eq:motorcargosde}) by simply replacing the tail force function $ \Fspring $ by a possibly motor-dependent tail force function $ \Fspring_i $, and replacing the linear force law (Eq.~\ref{eq:lintail}) by a more general form:
\begin{equation}
\Fspring (\slen) = \Fmag \Phip \left(\slen/L_c\right). \label{eq:nonlspring}
\end{equation}
Here $ \Phii (\xi) $ describes an arbitrary nondimensionalized spring potential, $L_c $ describes a length scale at which the spring force changes form, and $ \Fmag $ gives a characteristic magnitude of the tail force.  The linear spring force law (\ref{eq:lintail}) would correspond to a degenerate case in which $ \Phi (\xi) = \frac{1}{2} \xi^2 $, $ L_c $ is not well defined (and could be chosen arbitrarily), and $ \Fmag = \kappa L_c $.  For a tail force law that is linear at small displacements but nonlinear after some length scale $ L_c $ (such as the wormlike chain model mentioned above), $ \Phi $ would be a more general function than quadratic, and $ \Fmag $ could be reasonably chosen as $ \kappa L_c $, where $ \kappa $ gives the constant of proportionality of the force against displacement in the linear regime.  We will assume this type of model in the remaining discussion; other cases could be handled similarly with some modifications to the nondimensionalization procedure.

Proceeding with nonlinear force laws that can be written in the form (\ref{eq:nonlspring}) (with $ \Fmag = \kappa L_c$), and applying the same nondimensionalization as in Subsection~\ref{sec:nondim}, we obtain in place of Eq.~(\ref{eq:motorcargorescsde}):
\begin{align}
\difd \Xnd_i(\tt) &=  \epsilon g\left( s \nonld^{-1} \Phipi \left(\nonld (\Xnd_i(\tt) - \Znd (\tt))\right)\right) \, \difd \tt + \sqrt{\epsilon \rho}  \, \difd W_i(\tt)\\
	\difd \Znd (\tt) &= \sum_{i=1}^{\Nmot} \left[\nonld^{-1} \Phipi \left(\nonld (\Xnd_i(\tt) - \Znd (\tt))\right)- \frac{\Ftrnd}{\Nmot} \right] \, \difd \tt  +  \difd W_z(\tt),
\end{align}
with an additional nondimensional parameter defined as the ratio of the typical tail extension through thermal fluctuations to the length scale at which nonlinearity becomes important:
\begin{equation*}
\nonld \equiv \frac{\sqrt{2 \kB T/\kappa}}{L_c}
\end{equation*}
The dimensional analysis in Subsection~\ref{sec:nondim} implicitly assumes this quantity is not large ( $ \nonld \lesssim 1 $); otherwise using the length scale $ \sqrt{2 \kB T/\kappa} $ to characterize the typical tail length amounts to an inconsistent application of the linear spring force law beyond the regime of its validity.

The stochastic averaging proceeds as before, provided that the $ \Phi_i (\xi) $ satisfy some technical growth and smoothness conditions (namely,  $\Phi_i$ should be $C^1$ and $\int_0^x \exp \left\{2 \int_0^y \Phi_i'(v) dv \right\} dy \rightarrow \pm \infty$ as $|x| \rightarrow \infty$ and $\int_{-\infty}^\infty \exp \left\{-2 \int_0^x \Phi_i'(v) dv \right\} dx < \infty$), to again yield the averaged equations Eq.~(\ref{eq:motavg}), with the drift function for the motor coordinates and stationary distribution of the cargo modified as follows:
\begin{align*}
\gavg_i(\vec{x};\Ftrnd) &= \int_{\mathbb{R}} g\left[s \nonld^{-1} \Phipi \left(\nonld (x_i - z)\right)\right] m_{\vec{x},\Ftrnd}(z) \, \difd z; \label{eq:g-avg}\\ 
m_{\vec{x},\Ftrnd}(z) &= \normz^{-1} (\vec{x},\Ftrnd) \exp\left\{ - 
2\left[\Ftrnd z + \nonld^{-2} \sum_{i=1}^{\Nmot} \Phii \left(\nonld (x_i - z)\right) \right]\right\}, \\
\normz (\vec{x},\Ftrnd) &\equiv \int_{\mathbb{R}} 
\exp\left\{ - 
2\left[\Ftrnd z + \nonld^{-2} \sum_{i=1}^{\Nmot} \Phii \left(\nonld (x_i - z)\right) \right]\right\}.
\end{align*}

The analysis in Sections~\ref{sec:low-viscosity} and~\ref{sec:analysis-two-motors-fluct} for developing the effective drift and diffusivity for a cargo bound to  $ \Nmot = 1 $ 
or $ \Nmot = 2 $ motors proceeds in the same way, with just a few resulting changes in the formulas.  The effective drift (\ref{eq:velocity-one-motor}) of a single motor now involves a generally non-Gaussian average of the instantaneous force-velocity function:
\begin{align*}
\vf{1}{\Ftrnd} &= \normz^{-1} \int_{\mathbb{R}} g\left[s \nonld^{-1} \Phip \left(\nonld (x - z)\right)\right]  
 \exp\left\{-2\left[\Ftrnd z +
\nonld^{-2}  \Phi \left(\nonld (x-z)\right)  \right]\right\}\, \difd z, \\
\normz &= \int_{\mathbb{R}}  \exp\left\{-2\left[\Ftrnd z +
\nonld^{-2}  \Phi \left(\nonld (x-z)\right)  \right]\right\}\, \difd z.
\end{align*}
The formulas for the effective drift and diffusivity of a cargo bound to two motors are expressed in terms of the functions $ G_{\pm} (r;\Ftrnd) $, which would be redefined for nonlinear force laws as:
\begin{align*}
G_{\pm} (r;\Ftrnd) &\equiv G_1 (r;\Ftrnd) \pm G_2 (r;\Ftrnd), \\
\gavg_i (x_1,x_2;\Ftrnd) &= G_i (x_1-x_2;\Ftrnd).
\end{align*}

\begin{small}
\bibliographystyle{plain}
\bibliography{motors}
\end{small}

\end{document}

%% file: preamble.tex
\usepackage[table]{xcolor}
\usepackage{colortbl}
\usepackage{ctable}
\usepackage{tabularx}
\newcommand{\oh}{\frac{1}{2}}

\def\rbb{\mathbb{R}}
\newcommand{\Nmot}{N}

\newcommand{\difd}{\mathrm{d}}

\newcommand{\sigfun}{h}

\newcommand{\bZ}{Z}

\newcommand{\Fspring}{F}
\newcommand{\Fmag}{\bar{F}}
\newcommand{\slen}{r}
\newcommand{\lcon}{\ell_c}

\newcommand{\Ftrap}{\theta}
\newcommand{\Ftrnd}{\tilde{\theta}}

\newcommand{\kB}{k_{\mathrm{B}}}

\newcommand{\Xnd}{\tilde{X}}

\newcommand{\Znd}{\tilde{Z}}
\newcommand{\Xndd}{\check{X}^{(d)}}
\newcommand{\Yndd}{\check{Y}^{(d)}}

\newcommand{\Xna}{\check{X}}
\newcommand{\Yna}{\check{Y}}

\def\sigmc{\sigma_{\text{m/c}}}
\newcommand{\Xavg}{\bar{X}}
\newcommand{\tavg}{\bar{t}}
\newcommand{\gavg}{\bar{g}}

\newcommand{\gaussbig}[3]{\mathcal{A}[#3](#1,#2)}
\newcommand{\gaussbigbig}[3]{\mathcal{A}[#3]\Big(#1,#2\Big)}
\def\gtwoavg{G}

\newcommand{\Driftsys}{V_{\text{eff}}}

\newcommand{\tnd}{\tilde{t}}

\renewcommand{\vec}{\mathbf}
\newcommand{\M}{M}
\newcommand{\R}{R}
\newcommand{\rnf}{R}
\newcommand{\mnf}{M}
\newcommand{\xnf}{\hat X}
\newcommand{\znf}{\hat Z}
\newcommand{\tnf}{\bar t}
\newcommand{\dnf}[2]{\mathcal{\hat D}_{#1}(#2)}
\newcommand{\vnf}[2]{\mathcal{\hat V}_{#1}(#2)}
\newcommand{\df}[2]{\mathcal{D}_{#1}(#2)}
\newcommand{\vf}[2]{\mathcal{V}_{#1}(#2)}

\newcommand{\vforig}[2]{V_{#1}(#2)}

\newcommand{\xnfvec}{\mathbf{\hat X}}

\renewcommand{\P}{\mathbb{P}}

\newcommand{\gsym}{\eta}
\newcommand{\Phii}{\Phi_i}
\newcommand{\Phipi}{\Phi^{\prime}_i}
\newcommand{\nonld}{\lambda}
\newcommand{\normz}{K_Z}

\newcommand{\vmax}{v_{\mathrm{max}}}
\newcommand{\vmin}{v_{\mathrm{min}}}

\newcommand{\ddtavg}{\frac{\difd}{\difd \tavg}}